\documentclass[11pt,oneside]{amsart}
\usepackage{latexsym, amssymb, stmaryrd, eucal}

\newtheorem{thm}{Theorem}[section]
\newtheorem{lemma}[thm]{Lemma}
\newtheorem{prop}[thm]{Proposition}
\newtheorem{cor}[thm]{Corollary}

\newtheorem{fact}[thm]{Fact}
\newtheorem{rmk}[thm]{Remark}

\newtheorem{exercise}[thm]{Exercise}

\theoremstyle{definition}
\newtheorem{df}[thm]{Definition}
\newtheorem{ex}[thm]{Example}

\newtheorem{question}[thm]{Question}

\newcommand{\met}{\operatorname{met}}

\newcommand{\HF}{\operatorname{HF}}

\newcommand{\BN}{\mathbb{N}}

\newcommand{\cu}[1]{\mathcal{#1}}

\makeatletter

\def\dotminussym#1#2{%
  \setbox0=\hbox{$\m@th#1-$}%
  \kern.5\wd0%
  \hbox to 0pt{\hss\hbox{$\m@th#1-$}\hss}%
  \raise.6\ht0\hbox to 0pt{\hss$\m@th#1.$\hss}%
  \kern.5\wd0}
\newcommand{\dotminus}{\mathbin{\mathpalette\dotminussym{}}}

\def\dotlesym#1#2{%
  \setbox0=\hbox{$\m@th#1-$}%
  \kern.5\wd0%
  \hbox to 0pt{\hss\hbox{$\m@th#1\le$}\hss}%
  \raise 1\ht0\hbox to 0pt{\hss$\m@th#1.$\hss}%
  \kern.5\wd0}
\newcommand{\dotle}{\mathbin{\mathpalette\dotlesym{}}}

\begin{document}

\title{Continuous Sentences Preserved Under Reduced Products}

\author{Isaac Goldbring}

\address{Department of Mathematics, Statistics, and Computer Science\\
Department of Mathematics\\University of California, Irvine, 340 Rowland Hall (Bldg.\# 400),
Irvine, CA 92697-3875}
\email{isaac@math.uci.edu}
\thanks{Goldbring's work was partially supported by NSF CAREER grant DMS-1349399.}

\author{H. Jerome Keisler}

\address{Department of Mathematics\\
University of Wisconsin, 480 Lincoln Drive, Madison WI 53706, USA}
\email{keisler@math.wisc.edu}

\date{\today}

\begin{abstract}
Answering a question of Cif\'u Lopes, we give a syntactic characterization of those continuous sentences that are preserved under reduced products of metric structures.  In fact, we settle this question in the wider context of general structures as introduced by the second author.
\end{abstract}

\maketitle

\section{Introduction}

Reduced products are a generalization of ultraproducts where one works with an arbitrary proper filter instead of an ultrafilter.  Since the \L os' theorem fails at this level of generality, it becomes an interesting question to understand when the truth of a sentence in a reduced product follows from the truth of the sentence in each of the factor structures; when this happens for a given sentence in all reduced products, we say that that the sentence is \emph{preserved under reduced products} or is a \emph{reduced product sentence}.
In recent years, there has been an increasing interest in reduced products in continuous model theory, especially in the model theory of operator algebras (see [FS] or [Gh]).
This suggests that a syntactical characterization of reduced product sentences in continuous model theory may be useful.

In [Ke65], the second author showed, assuming the continuum hypothesis, that a sentence is preserved under reduced products if and only if the sentence is equivalent to a \emph{Horn sentence}.  Here, we recall that a \emph{basic Horn formula} is a first order formula that is either a finite disjunction of negated atomic formulas, or a
disjunction of an atomic formula and finitely many negated atomic formulas, while a \emph{Horn formula} is a first order formula that is built from basic Horn formulas
using the connective $\wedge$ and the quantifiers $\forall, \exists$.  In his thesis (see also [Ga]), Galvin showed that statement ``$\sigma$ is a reduced product sentence'' is actually arithmetical, whence an absoluteness argument shows that the use of the continuum hypothesis above is unnecessary.

At the end of Section 3 of [Lo], Cif\'u Lopes defines reduced products of metric structures and asks if a metric analogue of the classical theorem holds, to wit:  is there a syntactic characterization of those continuous sentences that are preserved under reduced products.  To be clear, a continuous sentence $\sigma$ is said to be preserved under reduced products if its value in a reduced product is $0$ whenever its value in each of the factor structures is $0$.  In Corollary 3.10 of [Lo],  Lopes gives one family of sentences preserved under reduced products.

The continuous analogue of the class of Horn sentences is the syntactically defined class of conditional sentences. In  this article, we settle  Lopes' question by showing that a continuous sentence is preserved under reduced products if and only if it is equivalent to a countable set of conditional sentences.  In fact, we show that this result holds in the wider context of \emph{general} structures, a generalization of metric structures investigated by the second author in [Ke].  General structures have no uniform continuity requirement on the predicate and function symbols and, in fact, do not even require a distinguished metric on the universe at all.

The plan of our proof is as follows:  we first show, as in the classical case, that the result holds when assuming the continuum hypothesis.  Incidentally, this result already appeared in the monograph [CK1966] on an earlier version of continuous logic.  We take the opportunity here to supply that proof in the modern incarnation of continuous logic.

In order to eliminate the use of the continuum hypothesis, rather than try to mimic Galvin's absoluteness argument, we instead introduce a universal procedure that converts all general structures into (classical) first-order structures and vice-versa; this conversion process takes advantage of the more flexible framework provided by general structures.  We use this conversion process (as well as a similar process on theories), and the absoluteness of being a classical reduced product sentence, to prove the absoluteness of being a continuous reduced product sentence.

The paper is organized as follows.  After this introduction, Section 2 contains all the preliminary information on general structures.  Section 3 describes the reduced product construction for general structures and defines the key notion of conditional sentences; here we prove that conditional sentences are preserved under reduced products.  In Section 4, we use the continuum hypothesis to show that reduced product sentences are equivalent to a set of conditional sentences and finally, in Section 5, we use the aforementioned conversion procedure to remove the use of the continuum hypothesis.  We also prove a result characterizing reduced product sentences as those satisfying a ``semi-continuity'' condition in all reduced products.

The paper also includes three appendices:  Appendix A proves a version of the Feferman-Vaught theorem for general structures (which was proven for metric structures by Ghasemi in [Gh]); Appendix B discusses some easy preservation theorems for general structures; Appendix C uses the techniques from Section \ref{s-elim-CH} to prove the Keisler-Shelah theorem for general structures.

We thank James Hanson for allowing us to include his observation on the Keisler-Shelah theorem for general structures.

\section{Preliminaries}

We will assume that the reader is familiar with the model theory of metric structures as developed in the paper [BBHU], and we will freely use notation from that paper.  In this section we review the basic model theory of general $[0,1]$-valued structures as developed in [Ke], and the relationship between general structures, pre-metric structures, and metric structures; unless stated otherwise, all Facts appearing below refer to the paper [Ke].

The syntax for general structures is the same as for metric structures.
The space of truth values is the real interval $[0,1]$, with $0$ representing true and $1$ representing false.  The connectives are the continuous functions from $[0,1]^n$ into $[0,1]$ for $n\in\mathbb{N}$, and the quantifiers are $\sup$ and $\inf$.
A \emph{vocabulary} $V$ is a set of predicate, function, and constant symbols.  Terms and atomic formulas are as in first order
logic.  Formulas are built in the usual way from terms using connectives and quantifiers, and sentences are formulas with no free variables.

A \emph{general structure} $\cu{M}$ consist of a  vocabulary  $V$, a non-empty universe set $M$, a function $P^\cu{M}\colon M^n\to[0,1]$ for each predicate symbol $P$ of arity $n$ in $V$, a function $F^\cu{M}\colon M^n\to M$ for each function symbol $F$ of arity $n$ in $V$, and an element $c^\cu{\cu{M}}\in M$ for each constant symbol $c$ in $V$.  In a general structure, there is no uniform continuity requirement for the function and predicate symbols.

If $V_0\subseteq V$, and $\cu{M}_0$ is obtained from $\cu{M}$ by forgetting every symbol of $V\setminus V_0$, we call $\cu{M}$ an \emph{expansion} of $\cu{M}_0$ to $V$, and call $\cu{M}_0$ the \emph{$V_0$-part} of $\cu{M}$.  If $\cu M$ is a general structure and $A\subseteq M$,  the expansion of $\cu M$ obtained by adding a constant symbol for each element of $A$ is denoted by $\cu M_A$, and the constant symbols for the elements of $A$ are called \emph{parameters from} $A$.

The  \emph{truth value} of a formula $\varphi(\vec{x})$ at a tuple $\vec{a}$ in a general structure $\cu{M}$ is denoted by $\varphi^\cu{M}(\vec{a})$.  We follow [Ke] rather than [BBHU] by defining a theory with vocabulary $V$ to be a set of sentences in
 vocabulary $V$.  We often use the connective $r\dotminus s$, which is defined by $r\dotminus s=\max(r-s,0).$

We say that $\cu M$ is a \emph{(general) model} of a theory $T$, in symbols $\cu M\models T$, if $\cu M$ is a general structure in which each sentence in $T$ has truth value $0$.
Note that $\cu M\models(\varphi \dotminus \psi)$ if and only if $\varphi^{\cu M}\le \psi^{\cu M}$,
 so to improve readability we also denote $r\dotminus s$ by $r\dotle s$. Two theories
$S,T$ are said to be equivalent if they have the same general models, but two sentences $\varphi,\psi$ are said to be equivalent if $\varphi^{\cu{M}}=\psi^{\cu{M}}$ for every general structure $\cu{M}$.  Thus the singletons $\{\varphi\}$ and $\{\psi\}$ are equivalent if and only if
$$\{\cu{M}\colon \varphi^{\cu{M}}=0\}=\{\cu{M}\colon \psi^{\cu{M}}=0\},$$
while the sentences $\varphi, \psi$ are equivalent if only if the singletons $\{\varphi\dotminus\varepsilon\}, \{\psi\dotminus\varepsilon\}$ are equivalent for all $\varepsilon\in[0,1]$.

Hereafter, $\cu{M}, \cu{N}$, sometimes with subscripts, will denote general structures with universe sets $M,N$ and vocabulary $V$, and $S, T, U$ will denote sets of sentences with the vocabulary $V$.

The notions of substructure (denoted by $\subseteq$), union of chain, elementary equivalence (denoted by $\equiv$), elementary substructure and extension (denoted by $\prec$ and $\succ$), elementary chain,  and $\kappa$-saturation are defined as in [BBHU], but applied to general structures as well as metric structures.

\begin{fact}  \label{f-Skolem}  (Downward Lowenheim-Skolem)  For every general structure $\cu{M}$, there is a general structure $\cu{M}'\equiv\cu{M}$ such that $|M'|\le\aleph_0+|V|$.
\end{fact}

\begin{fact}  \label{f-elementary-chain}  The union of an elementary chain of general structures $\langle \cu{M}_\alpha\colon\alpha<\beta\rangle$ is an elementary extension of each $\cu{M}_\alpha$.
\end{fact}

By an \emph{embedding} $h\colon\cu{M}\to\cu{N}$ we mean a function $h\colon M\to N$  such that $h(c^\cu{M})=c^\cu{N}$ for each constant symbol $c\in V$, and for every $n$ and $\vec{a}\in M^n$, $h(F^\cu{M}(\vec{a}))=F^\cu{N}(h(\vec{a}))$ for every function symbol $F\in V$ of arity $n$, and $P^\cu{M}(\vec{a})=P^\cu{N}(h(\vec{a}))$ for every predicate symbol $P\in V$ of arity $n$.  We say that $\cu{M}$ is \emph{embeddable in} $\cu{N}$ if there is an embedding $h\colon\cu{M}\to\cu{N}$.  Note that the image of an embedding $h\colon\cu{M}\to\cu{N}$ is a substructure of $\cu{N}$.

We now review the notions from [Ke] of a reduced general structure, and an ultraproduct of general structures.

\begin{df}  \label{d-reduced}
For $a,b\in M$, we write $a\doteq^{\cu{M}} b$ if for every atomic formula $\varphi(x,\vec{z})$ and tuple $\vec{c}\in M^{|\vec z|}$, $\varphi^{\cu{M}}(a,\vec{c})=\varphi^{\cu{M}}(b,\vec{c}).$  The relation $\doteq^{\cu{M}}$ is called Leibniz equality.
$\cu{M}$ is \emph{reduced} if whenever $a\doteq^\cu{M} b$ we have $a=b$.
\end{df}

The \emph{reduction} of the general structure $\cu{M}$ is the reduced structure $\cu{N}$ such that $N$ is the set of equivalence classes of elements of
$M$ under $\doteq^{\cu{M}}$, and the mapping that sends each element of $M$ to its equivalence class is an embedding of $\cu{M}$ onto $\cu{N}$.
We say that $\cu{M},\cu{M}'$ are \emph{isomorphic}, in symbols $\cu {M}\cong\cu{M}'$, if there is an embedding from the reduction of $\cu{M}$ onto the reduction of $\cu{M}'$.

\begin{rmk}  \label{r-reduction}
\noindent\begin{itemize}
\item $\cong$ is an equivalence relation on general structures.
\item Every general structure is isomorphic to its reduction.
\item  If there is an embedding of $\cu{M}$ \emph{onto} $\cu{N}$, then $\cu{M}\cong\cu{N}$.
\item $\cu{M}\cong\cu{N}$ implies $\cu{M}\equiv\cu{N}$.
\end{itemize}
\end{rmk}

\begin{rmk}  \label{r-reduction-part}  Let $V^0\subseteq V$.

\begin{itemize}
\item[(i)] If two general structures are isomorphic, then their $V^0$-parts are isomorphic.
\item[(ii)] For every general structure $\cu{M}$, the $V^0$-part of $\cu{M}$, the $V^0$-part of the reduction of $\cu{M}$, the
reduction of the $V^0$-part of $\cu{M}$, and the reduction of the $V^0$-part of the reduction of $\cu{M}$, are all isomorphic to each other.
\end{itemize}
\end{rmk}

Recall that for any ultrafilter $\cu{D}$ over a set $I$ and function $g\colon I\to[0,1]$, there is a unique value $r=\lim_{\cu{D}} g$ in $[0,1]$ such that
for each neighborhood $Y$ of $r$, the set of $i\in I$ such that $g(i)\in Y$ belongs to $\cu{D}$.

\begin{df} \label{d-ultraproduct} Let $\cu{D}$ be an ultrafilter  over a set $I$ and $\cu{M}_i$ be a general structure for each $i\in I$.
The \emph{pre-ultraproduct} $\prod^{\cu{D}}\cu{M}_i$ is the  general structure $\cu{M}'=\prod^{\cu{D}}\cu{M}_i$  such that:
\begin{itemize}
\item $M'=\prod_{i\in I} M_i$, the cartesian product.
\item For each constant symbol $c\in V$, $c^{\cu{M}'}=\langle c^{\cu{M}_i}\rangle_{i\in I}$.
\item For each $n$-ary function symbol $G\in V$ and $n$-tuple $\vec{a}$ in $M'$,
$$G^{\cu{M}'}(\vec{a})=\langle G^{\cu{M}_i}(\vec{a}(i))\rangle_{i\in I}.$$
\item For each $n$-ary predicate symbol $P\in V$ and $n$-tuple $\vec{a}$ in $M'$,
$$P^{\cu{M}'}(\vec{a})=\lim_{\cu{D}}\langle P^{\cu{M}_i}(\vec{a}(i))\rangle_{i\in I}.$$
\end{itemize}

The \emph{ultraproduct} $\prod_{\cu{D}}\cu{M}_i$ is the reduction of the pre-ultraproduct $\prod^{\cu{D}}\cu{M}_i$.
For each $a\in M'$ we also let $a_{\cu{D}}$ denote the equivalence class of $a$ under $\doteq^{\cu{M}'}$.

When $\cu M_i=\cu M$ for all $i\in I$, the reduced product $\prod_{\cu D}\cu M$ is called an \emph{ultrapower} of $\cu M$ and is denoted by $\cu{M}^I/{\cu D}$.
\end{df}

\begin{fact} \label{f-Los}  Let $\cu{M}_i$ be a general structure for each $i\in I$, let $\cu{D}$ be an ultrafilter over $I$, and let
$\cu{M}=\prod_{\cu{D}}\cu{M}_i$ be the ultraproduct.  Then for each formula $\varphi$ and tuple $\vec{b}$ in the cartesian product $\prod_{i\in I} M_i$,
$$\varphi^{\cu{M}}(\vec{b}_{\cu{D}})=\lim_{\cu{D}}\langle\varphi^{\cu{M}_i}(\vec{b}_i)\rangle_{i\in I}.$$
\end{fact}

As usual, it follows that:

\begin{fact}  \label{f-compactness}  (Compactness)  If every finite subset of  $T$ has a general model, then $T$ has a general model.
\end{fact}

\begin{cor}  \label{c-compactness-approx}
Suppose $\varphi$ is a sentence in the vocabulary $V$, and  $T\models\{\varphi\}$.

\noindent \begin{itemize}
\item[(i)]  For every $r \in (0,1]$ there is a finite $T_0\subseteq T$ such that $T_0\models \varphi\dotle r$.
\item[(ii)] There is a countable $T_1\subseteq T$ with $T_1\models \{\varphi\}$.
\end{itemize}
\end{cor}

\begin{proof}  (i):  Suppose not.  Then for some $r\in(0,1]$, $T\cup\{r/2\dotle\varphi\}$ is finitely satisfiable.  By the Compactness Theorem,
$T\cup\{r/2\dotle\varphi\}$ has a general model, so $T\not\models\{\varphi\}$.

(ii):  By (i), for each positive $n\in\BN$ there is a finite subset $T_n$ of $T$ such that $T_n\models \{\varphi\dotle 1/n\}$.
Then $T_\infty=\bigcup_n T_n$ is a countable subset of $T$, and $T_\infty\models \{\varphi\}$.
\end{proof}

For an infinite cardinal $\kappa$, we say that a general structure $\cu{M}$ is \emph{$\kappa$-saturated} if for every set $A\subseteq M$ of
cardinality $|A|<\kappa$, every set of formulas in the vocabulary of $\cu M$ with one free variable and parameters from $A$
that is finitely satisfiable in $\cu{M}_A$ is satisfiable in $\cu{M}_A$.

\begin{rmk}  \label{r-reduction-sat}
$\cu M$ is $\kappa$-saturated if and only if the reduction of $\cu M$ is $\kappa$-saturated.
\end{rmk}

\begin{df}
By a \emph{special cardinal} we mean a cardinal $\kappa$ such that $2^\lambda\le\kappa$ for all $\lambda<\kappa$.
We say that $\cu{M}$ is \emph{special} if $|M|$ is an uncountable special cardinal and  $\cu{M}$ is the union of an elementary chain of structures
$\langle \cu{M}_\lambda\colon \lambda<|M|\rangle$ such
that each $\cu{M}_\lambda$ is $\lambda^+$-saturated.  $\cu{M}$ is \emph{$\kappa$-special} if $\kappa$ is special and $\cu{M}$ is the reduction of a special
structure of cardinality $\kappa$.
\end{df}

Note that every strong limit cardinal is special, and if $2^\lambda=\lambda^+$ then $2^\lambda$ is special.
Note also that every $\kappa$-special structure is reduced and has cardinality $\le\kappa$.

\begin{rmk}  \label{r-special-expansion}  If $\cu{M}$ is $\kappa$-special and  $V^0\subseteq V$, then the reduction of the $V^0$-part of $\cu{M}$ is $\kappa$-special.
\end{rmk}

\begin{fact}  \label{f-special-unique}  (Uniqueness Theorem for Special Models) If $T$ is complete and $\cu{M}, \cu{N}$ are $\kappa$-special models of $T$,
then $\cu{M}$ and $\cu{N}$ are isomorphic.
\end{fact}

\begin{fact} \label{f-special-exist}  (Existence Theorem for Special Models)
If $\kappa$ is a special cardinal and $\aleph_0+|V|<\kappa$, then every reduced structure $\cu{M}$ such that
$|M|\le\kappa$ has a $\kappa$-special elementary extension.
\end{fact}

Define $r\dotplus s = \min(r+s,1)$.  A connective $C\colon [0,1]^n\to[0,1]$ is called \emph{increasing} if $C(s_1,\ldots,s_n)\le C(t_1,\ldots,t_n)$ whenever $s_i\le t_i$ for all $i\le n$. Important examples of increasing connectives are $\max, \min,\dotplus$ and the unary connectives $s \dotminus \varepsilon$ for a fixed $\varepsilon\in(0,1)$.  We say that $T, U$ are $S$-\emph{equivalent}  if $S\cup T$ and $S\cup U$ have the same general models.  The following lemma is the analogue for continuous model theory of Lemma 3.2.1 in [CK2012].

We say that a continuous formula $\varphi$ is \emph{restricted} if $\varphi$ is built from atomic formulas using only the quantifiers $\sup, \inf$ and
the connectives $0$, $1$, $\min$, $\max$, $\dotminus$, $\dotplus$, $\cdot/2$.  Note that any dyadic rational number $r\in[0,1]$ can be built
in finitely many steps using the connectives $0, 1, \cdot/2$.  Note that there at most $|V|+\aleph_0$ restricted continuous sentences in the vocabulary $V$.

\begin{lemma}  \label{l-restricted-full}
Every set $T$ of continuous sentences is equivalent to a set of restricted continuous sentences.
\end{lemma}

\begin{proof}  Let $V$ be the vocabulary of $T$.  Then $|V|\le |T|+\aleph_0$.  Let $U$ be the set of restricted continuous sentences $\theta$ in the vocabulary
$V$ such that $T\models\{\theta\}$.
 Then $T\models U$.  Suppose $\cu N\models U$ but not $\cu N\models T$.  Then for some $\varphi\in T$ and some dyadic rational $r>0$
we have $\varphi^{\cu N}\ge r$.  By Theorem 6.3 of [BBHU] (whose proof works for general structures as well as metric structures), there is a restricted continuous
sentence $\theta$ in the vocabulary $V$ such that $|\theta^\cu M-\varphi^\cu M|\le r/4$ for every general structure $\cu M$.  $(\theta\dotle r/4)$ is also a
restricted continuous sentence. Since $T\models\{\theta\}$, we have $T\models\{\theta\dotle r/4\}$, so $(\theta\dotle r/4)\in U$.
But since $\varphi^{\cu N}\ge r$, we have $\theta^{\cu N}\ge r - r/4$, so $(\theta\dotle r/4)$ is not true in $\cu N$, contradicting the assumption that $\cu N\models U$.  We conclude that $U\models T$, so $T$ is equivalent to $U$.
\end{proof}

 A \emph{metric signature} $L$ over $V$ specifies a distinguished binary predicate symbol $d\in V$ for distance, and equips each predicate
 or function symbol $S\in V$ with a modulus of uniform continuity  $\triangle_S\colon(0,1]\to(0,1]$ with respect to $d$.

A \emph{pre-metric structure} $\cu M_+=(\cu M,L)$ for $L$ consists of a general structure $\cu{M}$ with  vocabulary $V$,
and a metric signature $L$ over $V$,  such that $d^{\cu{M}}$ is a pseudo-metric
on $M$, and for each predicate symbol $P$ and function symbol $F$ of arity $n$, $P^{\cu M}$ and $F^{\cu M}$
are uniformly continuous with the bounds specified by $L$.
A \emph{metric structure} for $L$ is a pre-metric structure for $L$ such that $(M,d^\cu{M})$ is a complete metric space.

Given a pre-metric structure $\cu M_+=(\cu M,L)$, we will call $\cu M$  the \emph{downgrade of} $\cu M_+$,
and call $\cu M_+$  the \emph{upgrade of} $\cu M$ to $L$.
Note that two different pre-metric structures can have the same downgrade,
because they may have different metric signatures.
A pre-metric structure $(\cu M,L)$ is said to be reduced if and only if its downgrade $\cu M$ is reduced.
Similarly for $\kappa$-saturated, etc.  Given a family $\langle(\cu M_i,L)\colon i\in I\rangle$ of pre-metric structures with the same signature $L$,
the ultraproduct is defined as the pre-metric structure $\prod_{\cu D}(\cu M_i,L):=(\prod_{\cu D}\cu M_i,L)$.

 We remind the reader that every pre-metric structure for $L$ has a unique completion up to isomorphism,
that this completion is a metric structure for $L$, and that every pre-metric structure is elementarily embeddable in its completion.

\begin{fact}  \label{f-reduced-premetric}  Let $(\cu M,L)$ be a  pre-metric structure with distinguished distance $d$.
$\cu M$ is reduced if and only if for all $x,y\in M$, if  $d^{\cu M}(x,y)=0$ then $x=y$.
\end{fact}

In particular, every metric structure is reduced.

\begin{rmk} \label{r-saturated-metric}  Every pre-metric structure for $L$ that is reduced and $\aleph_1$-saturated
 is a metric structure for $L$.
\end{rmk}

The preservation theorems in this paper will apply to general structures.   Using  Fact \ref{f-metric-theory} below, we will immediately get analogous results for pre-metric and metric structures.

\begin{fact}  \label{f-metric-theory}  (See [Ca], page 112.)  For each metric signature $L$ over $V$, there is a theory $\met(L)$ whose general models are exactly
the downgrades of pre-metric structures for $L$.  Thus every pre-metric structure for $L$ satisfies $\met(L)$.  Each sentence in $\met(L)$ consists of finitely many $\sup$ quantifiers followed by a quantifier-free formula.
\end{fact}

A \emph{metric theory} $(T,L)$ consists of a metric signature $L$ and a set $T$ of sentences in the vocabulary of $L$ such that $T\models \met(L)$.  We say that $T$ is a \emph{metric theory with signature} $L$ if $(T,L)$ is a metric theory.  A pre-metric (or metric) \emph{model} of a metric theory $(T,L)$ is a pre-metric (or metric) structure $\cu{M}_+=(\cu M,L)$  such that $\varphi^{\cu{M}}=0$ for all $\varphi\in T$.
Thus Fact \ref{f-metric-theory} shows that for any metric theory $(T,L)$ and general structure $\cu M$, $(\cu M,L)$ is a pre-metric model of $T$ if and only if $\cu M$ is a general model of $T$.

We say that a sequence $\langle\varphi_m(\vec x,\vec y)\rangle_{m\in\BN}$ of formulas is \emph{Cauchy} in $T$
if for each $\varepsilon >0$ there exists $m$ such that for all $k\ge m$,
$$ T\models \sup_{\vec x}\sup_{\vec y}|\varphi_m(\vec x,\vec y)-\varphi_k(\vec x,\vec y)|\dotle\varepsilon.$$

\begin{df} \label{d-metric-expansion}
Let $T$ be a theory in a vocabulary $V$, and let $D$ be a new binary predicate symbol.
We say that $T_e$ is a \emph{pre-metric expansion} of $T$ (with signature $L_e$) if:
\begin{itemize}
\item[(i)]  $(T_e,L_e)$ is a metric theory whose signature $L_e$ has vocabulary $V_D:=V\cup\{D\}$ and distance predicate $D$.
\item[(ii)] There is a Cauchy sequence $\langle d\rangle=\langle d_m\rangle$ of formulas in $T$ such that
the general models of $T_e$ are exactly the structures
of the form $\cu M_e=(\cu M,[\lim d_m]^{\cu M})$, where $\cu M$ is a general model of $T$.
\end{itemize}
\end{df}

\begin{fact}  \label{f-premetric-expansion}  Suppose $V$ has countably many predicate symbols and $T$ is a $V$-theory.

\noindent\begin{itemize}
\item[(i)] (Theorem 3.3.4 of [Ke].) $T$ has a pre-metric expansion.

\item[(ii)](Proposition 4.4.1 of [Ke].)  If $T_e$ a pre-metric expansion of $T$, $\cu M, \cu N\models T$, and $\cu M\equiv\cu N$, then $\cu M_e\equiv\cu N_e$.
\end{itemize}
\end{fact}

\section{Reduced Products}

As in the classical case, the reduced product construction is a generalization of the ultraproduct construction
with a proper filter $\cu F$ instead of an ultrafilter $\cu D$.  We refer to Sections 4.1 and 6.2 of [CK2012] for a treatment of reduced products in classical model theory.  In Section 7.4 of [CK1966], reduced products of general structures with a predicate symbol $\circeq$
for the discrete metric
were defined, and here we will modify this to define reduced products of general structures  without $\circeq$.  For the special case of metric structures, this will be exactly the definition of reduced product in the paper [Lo].

We say that a sentence $\psi$ is \emph{preserved under reduced products} if every reduced product of models of $\psi$  is also a model of $\psi$.  As mentioned in the introduction, [Lo] proposed the problem of characterizing the sentences that are preserved under reduced products of metric structures.
Theorem 7.4.23 of [CK1966] implies that, assuming the continuum hypothesis, a sentence $\psi$ is preserved under reduced products of general structures  if and only if $\{\psi\}$ is equivalent to  some countable set of conditional sentences (see Definition \ref{d-conditional} below).  In this paper we will present that theorem and its proof in the modern framework. It will follow (in Corollary \ref{c-main-metric}), still assuming the continuum hypothesis, that a sentence $\psi$ is preserved under reduced products of metric structures if and only if $\{\psi\}$ is $\met(L)$-equivalent to some countable set of conditional sentences.  In  Theorem \ref{t-conditional-preserve}, we will prove in ZFC
that every conditional sentence is preserved under reduced products.  In the next section, in Theorem \ref{t-main}, we will prove, assuming the continuum hypothesis, that every
sentence that is preserved under reduced products is equivalent to a countable set of conditional sentences.
Thus, to solve the problem posed in [Lo], all that has to be done is to eliminate the continuum hypothesis from the result in [CK1966].  We will carry out that elimination in  Section \ref{s-elim-CH}.

In this section, the letter $I$ will always denote a non-empty set, to be used as an index set, and $\cu{F}$ will denote a proper filter over $I$.  Let $\beta I$ be the set of all ultrafilters over $I$.  Recall that when $\cu{D}\in\beta I$ and $g\colon I\to[0,1]$, there is a unique value $r=\lim_\cu{D} g$ in $[0,1]$ such that for each neighborhood $Y$ of $r$, the set of $i\in I$ such that $g(i)\in Y$ belongs to $\cu{D}$.  The following topological lemma partly motivated the definition of reduced product that comes next.

\begin{fact}  \label{f-F-limit}  (Lemma 7.4.16 in [CK1966], with the order reversed.) For every  $g\colon I\to[0,1]$,
\begin{equation}  \label{eq1}
\sup\{\lim_\cu{D}g\colon \cu{F}\subseteq\cu{D}\in\beta I\}=\inf_{J\in\cu{F}}\sup_{i\in J}g(i).
\end{equation}
\end{fact}

\begin{proof}  Let $x$ denote the left side of (\ref{eq1}), $y$ denote the right side of (\ref{eq1}), and $z=(x+y)/2$.
Suppose first that $x>y$, so $x>z>y$.  Then there exists $\cu F\subseteq\cu D\in\beta I$ and $J\in\cu F$ such that
$$\lim_{\cu D} g > z > \sup_{i\in J} g(i).$$
But then $\{i\in I\colon g(i)>z\}\cap J \in\cu D$, so there exists $i\in J$ such that $g(i) > z$, a contradiction.

Now suppose $x < y$, so $x < z < y$. Then whenever $\cu F\subseteq \cu D\in\beta I$ and $J\in\cu F$ we have $J\in \cu D$ but
$$\lim_{\cu D} g < z < \sup_{i\in J}g(i),$$
which is again a contradiction.
\end{proof}

In [CK1966], the left side of equation (\ref{eq1}) is denoted by $\cu{F}$-$\sup x$, while in [Lo] the right side of (\ref{eq1}) is denoted by $\limsup_{\cu{F}} x$.  We will use the latter notation here. Thus
$$ \limsup_{\cu{F}}x := \inf_{J\in\cu{F}}\sup_{i\in J}x(i).$$

The left side of equation (\ref{eq1}) makes it clear that, for $g:I\to [0,1]$, we have $\limsup_{\cu{D}} g=\lim_{\cu{D}} g$ when $\cu{D}$ is an ultrafilter on $I$.

 We now define the reduced product $\prod_{\cu{F}}\cu{M}_i$ of an indexed family $\langle\cu{M}_i\colon i\in I\rangle$ modulo $\cu{F}$.  The reduced product, like the ultraproduct, will be constructed in two steps: first construct the pre-reduced product $\prod^{\cu{F}}\cu{M}_i$, whose universe is the cartesian product  $\prod_{i\in I}\cu{M}_i$, and then take the reduction.

\begin{df} \label{d-reduced-product} Let $\cu{M}_i$ be a general structure for each $i\in I$.  The \emph{pre-reduced product} $\prod^{\cu{F}}\cu{M}_i$ is the  general structure $\cu{M}'$ such that:
\begin{itemize}
\item $M'=\prod_{i\in I} M_i$, the cartesian product;
\item For each constant symbol $c\in V$, $c^{\cu{M}'}=\langle c^{\cu{M}_i}\colon i\in I\rangle$;
\item For each $n$-ary function symbol $G\in V$ and $n$-tuple $\vec{a}$ in $M'$,
$$G^{\cu{M}'}(\vec{a})=\langle G^{\cu{M}_i}(\vec{a}(i))\colon i\in I\rangle;$$
\item For each $n$-ary predicate symbol $P\in V$ and $n$-tuple $\vec{a}$ in $M'$,
$$P^{\cu{M}'}(\vec{a})=\limsup_{\cu{F}}\langle P^{\cu{M}_i}(\vec{a}(i))\colon i\in I\rangle.$$
\end{itemize}
The \emph{reduced product} $\prod_{\cu{F}}\cu{M}_i$ is the reduction of $\prod^{\cu{F}}\cu{M}_i$.  We also let $a_{\cu{F}}$ denote the equivalence class of $a$ under $\doteq^{\cu{M}'}$.
\end{df}

Thus an ultraproduct is  a reduced product modulo an ultrafilter.

\begin{rmk}  \label{r-atomic-reduced-product} Let $\cu{M}'=\prod^{\cu{F}}\cu{M}_i$  and $\cu{N}=\prod_{\cu{F}}\cu{M}_i$. The mapping $a\mapsto a_{\cu{F}}$ is an embedding of $\cu{M}'$ onto  $\cu{N}$.  Therefore:
\begin{itemize}
\item $N=\{a_{\cu{F}}\colon a\in M'\}$.
\item For each term $t(\vec{x})$ and tuple $\vec{a}$ in $M'$, $t^{\cu{N}}(\vec{a}_{\cu{F}})=(t^{\cu{M}'}(\vec{a}))_{\cu{F}}.$
\item For each atomic formula $\alpha(\vec{x})$ and tuple $\vec{a}$ in $M'$,
$$\alpha^{\cu{N}}(\vec{a}_{\cu{F}})=\alpha^{\cu{M}'}(\vec{a})=\limsup_{\cu{F}}\langle \alpha^{\cu{M}_i}(\vec{a}(i))\colon i\in I\rangle.$$
\item If $\cu{M}_i\cong\cu{N}_i$ for each $i\in I$, then $\prod_{\cu{F}}\cu{M}_i\cong \prod_{\cu{F}}\cu{N}_i.$
\item If $\cu{N}_i$ is the reduction of $\cu{M}_i$ for each $i\in I$, then $\prod_{\cu{F}}\cu{M}_i\cong \prod_{\cu{F}}\cu{N}_i.$
\end{itemize}
\end{rmk}

\begin{rmk}  \label{r-part-reduced-product}  Let $V^0\subseteq V$, and
let $\cu{M}^0_i$ be the $V^0$-part of $\cu{M}_i$ for each $i\in I$.  Then $\prod^{\cu{F}}\cu{M}^0_i$ is the $V^0$-part of $\prod^{\cu{F}}\cu{M}_i$, and $\prod_{\cu{F}}\cu{M}^0_i$ is isomorphic to the $V^0$-part of $\prod_{\cu{F}}\cu{M}_i$.
\end{rmk}

\begin{proof}  Let $\cu{M}=\prod^{\cu{F}}\cu{M}_i$, and $\cu{M}^0=\prod^{\cu{F}}\cu{M}^0_i$.  By definition, $\prod_{\cu{F}}\cu{M}_i$ is the reduction of $\cu{M}$, and $\prod_{\cu{F}}\cu{M}^0_i$ is the reduction of $\cu{M}^0$.  It is clear that $\cu{M}^0$ is the $V^0$-part of $\cu{M}$.
Therefore, by Remark \ref{r-reduction-part} (ii), $\prod_{\cu{F}}\cu{M}^0_i$ is isomorphic to the $V^0$-part of $\prod_{\cu{F}}\cu{M}_i$.
\end{proof}

We next define the conditional sentences, which are the continuous analogues of  Horn sentences.

\begin{df} \label{d-conditional} A formula $\varphi$ is \emph{primitive conditional} if there are atomic formulas $\alpha_0,\ldots,\alpha_n$ and unary increasing connectives $C_0,\ldots,C_n$ such that
\begin{equation}  \label{eq2}
\varphi=\min(C_0(\alpha_0),C_1(1-\alpha_1),\ldots,C_n(1-\alpha_n)).
\end{equation}
The set of \emph{conditional formulas} is the least set of formulas that contains the primitive conditional formulas and is closed under the application of the $\max$ connective and the quantifiers $\sup,\inf$.
\end{df}

By taking $C_0=1$ in (\ref{eq2}), one can see that $\min(C_1(1-\alpha_1),\ldots,C_n(1-\alpha_n))$ is also a primitive conditional formula.

\begin{lemma} \label{l-TL-conditional}  The set $\met(L)$ of axioms for pre-metric structures with signature $L$ is equivalent to a set of  conditional sentences.
\end{lemma}

\begin{proof}  The property $d(x,y)\dotle d(y,x)$ is equivalent to the countable set of primitive conditional formulas
$$\{ (r\dotle d(x,y))\dotle (r\dotle d(y,x))\colon r \mbox{ dyadic rational} \}.$$
Using a similar trick, the property
$$d(x,z)\dotle d(x,y)\dotplus d(y,z)$$
 and the uniform continuity property
$$\max_{k\le n} d(x_k,y_k)< \delta \Rightarrow |P(\vec x) - P(\vec y)|\le \varepsilon$$
can be expressed by countable sets of primitive conditional formulas.  The sentences in $\met(L)$ can then be expressed by
countable sets of sentences obtained by putting $\sup$ quantifiers in front of primitive conditional formulas.
\end{proof}

\begin{lemma}  \label{l-conditional-minus-epsilon}
For every (primitive) conditional formula $\varphi$ and every increasing unary connective $B$,  $B\varphi$ is equivalent to a (primitive) conditional formula.
\end{lemma}

\begin{proof}  If $\varphi$ is primitive conditional, then $B\varphi$ is equivalent to the primitive conditional formula obtained from equation (\ref{eq2}) by replacing each connective $C_k$ by $B\circ C_k$.  The result for arbitrary conditional formulas $\varphi$ follows  by an easy induction on the complexity of $\varphi$.
\end{proof}

The paper [Lo] considered other classes of formulas that are obtained by closing the class of atomic formulas under $\inf, \sup$, and certain connectives.  These classes are quite different from the class of conditional formulas, since they do not contain the primitive conditional formulas.

The following topological fact is motivated by Definition \ref{d-conditional}.

\begin{fact}  \label{f-basic}  (Lemma 7.4.20 in [CK1966].) Let $y_0,\ldots,y_n\colon I\to[0,1]$, and $C_0,\ldots,C_n$ be increasing unary connectives.  Suppose
$$J':= \{i\in I\colon \min(C_0(y_0(i)),C_1(1-y_1(i)),\ldots,C_n(1-y_n(i)))=0\}\in\cu{F}.$$
Then
$$\min(C_0(\limsup_{\cu{F}}y_0),C_1(1-\limsup_{\cu{F}}y_1),\ldots,C_n(1-\limsup_{\cu{F}}y_n))=0.$$
\end{fact}

\begin{proof}  We may assume that $C_k(0)=0$ for $1\le k\le n$, because otherwise we may remove $C_\ell$ when $\ell$ is the least $\ell\ge 1$ such
that $C_\ell(0)>0$.
For $1\le k\le n$, let $z_k=\sup\{z\colon C_k(z)=0\}$.  Since each $C_k$ is increasing and continuous, for $1\le k\le n$ we have $C_k(z)=0$ if and only if $z\le z_k$.

Suppose that  $C_k(1-\limsup_{\cu F} y_k)>0$ for each $1\le k\le n$.  We prove that $C_0(\limsup_{\cu F}y_0)=0$.
Fix $1\le k\le n$.  Then $1-\limsup_{\cu F}y_k > z_k$, so
$$\inf_{J\in\cu F}\sup_{i\in J}y_k(i)=\limsup_{\cu F} y_k< 1-z_k.$$
Hence there exists $J\in\cu F$ such that for every $i\in J$ and $1\le k\le n$, $y_k(i)<1-z_k$, so $1-y_k(i) > z_k$ and $C_k(1-y_k(i))>0$.
Then $J'\cap J\in\cu F$, and therefore
$$J'\cap J\subseteq \{i\in I\colon C_0(y_0(i))=0\}\in\cu F.$$
Hence there is a set $J''\in\cu F$ such that $C_0(y_0(i))=0$ for all $i\in J''$.
Then there exists $z$ such that  $C_0(z)=0$.  Since $C_0$ is increasing and continuous, we have
$\{z\colon C_0(z)=0\}=[0,z_0]$ for some $z_0\in[0,1]$.  So $y_0(i)\le z_0$ for all $i\in J''$, and thus $\limsup_{\cu F} y_0\le z_0$.  It follows that $C_0(\limsup_{\cu F} y_0)=0$,  as required.
\end{proof}

\begin{thm}  \label{t-conditional-preserve}  (Exercise 7K in [CK1966].)  If $T$ is equivalent to a set of conditional sentences, then every reduced product of general models of $T$ is a general model of $T$ (that is, $T$ is preserved under reduced products of general structures).
\end{thm}

\begin{proof} It suffices to prove the result when $T=\{\psi\}$ for a single conditional sentence $\psi$.  Consider an indexed family $\langle\cu{M}_i\colon i\in I\rangle$ of general structures, and let $\cu{N}=\prod_{\cu{F}}\cu{M}_i$ be the reduced product.  First suppose that $\varphi$ is a primitive conditional formula as in (\ref{eq2}) above, and let $\vec{a}$ be a tuple of elements of the cartesian product $\prod_{i\in I} \cu{M}_i$.  For $k\le n$ put $y_k(i)=\alpha_k^{\cu{M}_k}(\vec{a}(i))$.  Then Fact \ref{f-basic} shows that
$$\{i\in I\colon \varphi^{\cu{M}_i}(\vec{a}(i))=0\}\in\cu{F}\Rightarrow\varphi^{\cu{N}}(\vec{a}_{\cu{F}})=0.$$
By Lemma \ref{l-conditional-minus-epsilon}, for each $\varepsilon\in [0,1]$, $\varphi\dotminus\varepsilon$ is equivalent to a primitive conditional formula.  Moreover, for every $\cu{M}$, tuple $\vec{a}$ in $M$, and $\varepsilon$, $\varphi^{\cu{M}}(\vec{a})\le\varepsilon$ if and only if $\varphi^{\cu{M}}(\vec{a})\dotminus\varepsilon=0.$
Therefore for each $\varepsilon\in[0,1]$,
\begin{equation}  \label{eq3}
\{i\in I\colon \varphi^{\cu{M}_i}(\vec{a}(i))\le\varepsilon\}\in\cu{F}\Rightarrow\varphi^{\cu{N}}(\vec{a}_{\cu{F}})\le\varepsilon.
\end{equation}
We now complete the proof of the theorem by proving the following stronger statement: for every conditional formula $\psi$ and $\varepsilon\in[0,1]$,
\begin{equation}  \label{eq4}
\{i\in I\colon \psi^{\cu{M}_i}(\vec{a}(i))\le\varepsilon\}\in\cu{F}\Rightarrow\psi^{\cu{N}}(\vec{a}_{\cu{F}})\le\varepsilon.
\end{equation}  The base case of the induction is taken care of by equation (\ref{eq3}).  For the inductive step, we only treat the $\inf$ case, the remaining cases being similar.  Suppose  (\ref{eq4}) holds for $\theta(\vec{x},y)$, and
$\psi(\vec{x})=\inf_y\theta(\vec{x},y)$.  Assume that
$$\{i\in I\colon\psi^{\cu{M}_i}(\vec{a}(i))\le\varepsilon\}\in\cu{F}.$$
Then for each $\delta>0$ there exists $b\in\prod_{i\in I}M_i$ such that
$$\{i\in I\colon\theta^{\cu{M}_i}(\vec{a}(i),b(i))\le\varepsilon+\delta\}\in\cu{F}.$$
By inductive hypothesis, $\theta^{\cu {N}}(\vec{a}_{\cu{F}},b_{\cu{F}})\le\varepsilon+\delta.$  $\psi^{\cu{N}}(\vec{a}_{\cu{F}})\le\epsilon+\delta$.
Since this holds for all $\delta>0$, $\psi^{\cu{N}}(\vec{a}_{\cu{F}})\le\varepsilon$, as required.
\end{proof}

  It is shown in  [Lo], pages 221-222, that every reduced product of metric structures is a metric structure.
We have analogous results for pre-metric structures, and for $V^0$-parts of pre-metric structures.

\begin{cor}  \label{c-metric-product}  Suppose $L$ is a metric signature over $V$ with distinguished distance $d$, $d\in V^0\subseteq V$, and
$L^0$ is the restriction of $L$ to $V^0$.

\begin{itemize}
\item[(i)]  Any reduced product of pre-metric structures (with signature $L$)  is a pre-metric structure.
\item[(ii)] If the $V^0$-part of $\cu{M}_i$ is a pre-metric structure (with signature $L^0$) for each $i\in I$ , then the $V^0$-part of $\prod_{\cu{F}}\cu{M}_i$ is a pre-metric structure.
\item[(iii)]   If the reduction of the $V^0$-part of  $\cu{M}_i$ is a metric structure (with signature $L^0$) for each $i\in I$, then the reduction of the $V^0$-part of $\prod_{\cu{F}}\cu{M}_i$ is a metric structure.
\end{itemize}
\end{cor}

\begin{proof} (i) and (ii):  By Fact \ref{f-metric-theory}, Lemma \ref{l-TL-conditional}, and Theorem \ref{t-conditional-preserve}.

(iii):   For each $i\in I$, let $\cu{M}^0_i$ be the reduction of the $V^0$-part of $\cu{M}_i$, and let $\cu M^0$ be the reduction of $\prod_{\cu F} \cu M^0_i$.
By Remarks \ref{r-atomic-reduced-product} and  \ref{r-part-reduced-product}, $\cu M^0$ is isomorphic to the $V^0$-part of $\prod_{\cu{F}} \cu{M}_i$.
By hypothesis, each $\cu{M}^0_i$ is a metric structure.  By (ii), $\cu M^0$ is a pre-metric structure, and its distinguished distance $d^{\cu M}$ is a metric.
As in [Lo], one can show that $d^{\cu M}$ is a complete metric by modifying the proof of Proposition 5.3 in [BBHU].
\end{proof}

The following corollary will be used later to apply results about general structures to metric structures.

\begin{cor}  \label{c-product-completion}
Suppose $\cu F$ is a proper filter over a set $I$, $L$ is a metric signature, and for each $i\in I$, $(\cu N_i,L)$ is a pre-metric structure and $(\cu M_i,L)$ is its completion.
Then $(\prod_{\cu F}\cu M_i,L)$ is the completion of $(\prod_{\cu F} \cu N_i,L)$.
\end{cor}

\begin{proof}
For each $i$,  $(\cu M_i,L)$ is a metric structure.  By Corollary \ref{c-metric-product}, $(\prod_{\cu F}\cu N_i,L)$ is a pre-metric structure, and
$(\prod_{\cu F}\cu M_i,L)$ is a metric structure.  For each $i\in I$, $\cu N_i$ is a dense substructure of $\cu M_i$.  Therefore, for each $a\in\prod_{i\in I} M_i$ and
$\varepsilon>0$ there exists $b\in\prod_{i\in I} N_i$ such that $\cu M_i\models d(a_i,b_i)\dotle\varepsilon$ for all $i\in I$.  The formula $d(x,y)\dotle\varepsilon$
is conditional, so by Theorem \ref{t-conditional-preserve} we have $\prod_{\cu F}\cu M_i\models d(a_{\cu F},b_{\cu F})\dotle\varepsilon$.  Therefore $\prod_{\cu F}\cu N_i$
is dense in  $\prod_{\cu F}\cu M_i$.  By definition, all reduced products are reduced structures.  It follows that $(\prod_{\cu F}\cu M_i,L)$ is the completion of  $(\prod_{\cu F}\cu N_i,L)$.
\end{proof}

Recall from the introduction that a \emph{basic Horn formula} is a first order formula that is either a finite disjunction of negated atomic formulas, or a
disjunction of an atomic formula and finitely many negated atomic formulas while a \emph{Horn formula} is a first order formula that is built from basic Horn formulas
using the connective $\wedge$ and the quantifiers $\forall, \exists$.
We will use the following result of Galvin, which is a result in ZFC that characterizes the first order sentences preserved under reduced products

\begin{fact}  \label{f-Horn}  ([Ga], Theorems 3.4, 5.2, and Section 6).

(i) A first order sentence $\theta$ is preserved under reduced products if and only if
$\theta$ is equivalent to a Horn sentence.

(ii)  A first order theory  is preserved under reduced products if and only if it is equivalent to a set of Horn sentences.

\end{fact}

Each first order formula $\theta$ can be converted to a continuous formula $\theta^c$ by replacing the first order connectives $\wedge, \vee,  \neg$ by the continuous connectives $\min, \max, 1\dotminus\cdot$, and replacing the first order quantifiers $\forall,\exists$ by the continuous quantifiers $\sup, \inf$.
The following gives a connection between preservation under reduced products in first order and continuous model theory.

\begin{cor} \label{c-FO-vs-cont}
Let $\theta$ be a first order sentence.  Each of the following conditions implies the next.

\begin{itemize}
\item[(i)] $\theta$ is a Horn sentence.
\item[(ii)] $\theta^c$ is a conditional sentence.
\item[(iii)]  $\theta^c$ is preserved under continuous reduced products.
\item[(iv)]  $\theta$ is equivalent to a Horn sentence.
\end{itemize}
\end{cor}

\begin{proof}
(i) $\Rightarrow$ (ii) is clear from the definition.  By Theorem \ref{t-conditional-preserve},  (ii) $\rightarrow$ (iii).

Assume (iii).  Each first order structure $\cu M$ without identity may be considered as a general structure in which every formula has truth values in $\{0,1\}$.
By Remark 2.3 in [Lo], if each $\cu M_i$ is a first order structure, then the reduced product $\prod_{\cu F} \cu M_i$ considered as a first order structure
is the same as the reduced product $\prod_{\cu F} \cu M_i$ considered as a general structure.  Therefore $\theta$ is preserved under first order reduced products,
so by Fact \ref{f-Horn}, (iv) holds.
\end{proof}

\begin{ex}  \label{e-horn-counterexample}
In Corollary \ref{c-FO-vs-cont}, Condition (iv) does not imply Condition (iii).
\end{ex}

\begin{proof}  For example, let $P, Q$ be $0$-ary predicates, and let $\theta$ be
the first order sentence $P \vee Q\vee\neg(P\vee Q)$.  Then $\theta$ is equivalent to the Horn sentence $P\vee\neg P$,  but
we will show that $\theta^c$ is not preserved under continuous reduced products.
$\theta^c$ is the continuous sentence
$$\min(P, Q, 1\dotminus(\min(P,Q))).$$
 Let $I=\{1,2\}$ and let $\cu F$ be the filter $\cu F=\{I\}$ (so reduced products modulo $\cu F$ are direct products).  Let $0<r\le1/2$.
Let $\cu M_1$ be a general structure such that $\cu M_1\models P=r$ and $\cu M_1\models Q=0$.  Let $\cu M_2$  be a general structure such that $\cu M_2\models P=0$ and
$\cu M_2\models Q=r$.  We leave it to the reader to check that $\cu M_1\models\theta=0$ and $\cu M_2\models\theta=0$, but $\prod_{\cu F}\cu M_i\models\theta=r$.
\end{proof}

We will need the following result of Ghasemi [Gh], which follows from a metric analogue of the Feferman-Vaught theorem in first order model theory.

\begin{fact} \label{f-Gh}  (Proposition 3.6 in  [Gh].)
If $\cu F$ is a proper filter over $I$, and $\cu M_i, \cu N_i$ are metric structures with $\cu M_i\equiv \cu N_i$ for each $i\in I$, then
$\prod_{\cu F}\cu M_i\equiv \prod_{\cu F}\cu N_i$.
\end{fact}

\begin{rmk}\label{pre-Gh}
Fact \ref{f-Gh} was stated in [Gh] only for metric structures, but the proof shows that it holds for pre-metric structures as well.
\end{rmk}

Here is the analogous result for general structures.

\begin{prop}  \label{p-Gh-general}
Suppose $V$ is countable, $\cu F$ is a proper filter over $I$, and $\cu M_i, \cu N_i$ are general structures with $\cu M_i\equiv \cu N_i$ for each $i\in I$. Then
$\prod_{\cu F}\cu M_i\equiv \prod_{\cu F}\cu N_i$.
\end{prop}

\begin{proof}
By Fact \ref{f-premetric-expansion} (i), the empty theory $T$ has a pre-metric expansion $T_e$.  For each $i\in I$, $\cu M_{ie}$ and $\cu N_{ie}$ are pre-metric structures
with signature $L_e$.  By Corollary \ref{c-metric-product}, $\prod_{\cu F}(\cu M_{ie})$ and $\prod_{\cu F}(\cu N_{ie})$ are also   pre-metric structure with signature $L_e$.
By Fact \ref{f-premetric-expansion} (ii), $\cu M_{ie}\equiv\cu N_{ie}$ for each $i$.  Then by Remark \ref{pre-Gh},
$\prod_{\cu F}(\cu M_{ie})\equiv\prod_{\cu F}(\cu N_{ie})$.  Finally, taking the $V$-parts and using Remark \ref{r-part-reduced-product}, we have
$\prod_{\cu F}\cu M_{i}\equiv\prod_{\cu F}\cu N_{i}$.
\end{proof}

\section{Consequences of the Continuum Hypothesis}

In this section we will show  that,  assuming  $2^\lambda=\lambda^+$  for some $\lambda\ge |V|+\aleph_0$,
the converse of Theorem \ref{t-conditional-preserve} holds.
In the next section we will eliminate the assumption $2^\lambda=\lambda^+$ from that result.
The hypothesis $2^\lambda=\lambda^+$ is called the generalized continuum hypothesis (GCH) at $\lambda$.

In the following, given an infinite set $I$, we say that a statement holds for \emph{most} $i\in I$ if it holds for all but fewer than $|I|$ elements $i\in I$.

\begin{lemma} \label{l-conditional-most}  (Lemma 7.4.22 in [CK1966].)  Assume that:
\begin{itemize}
\item $\lambda$ is infinite and the GCH at $\lambda$ holds, $2^\lambda=\lambda^+$.
\item $|I|=\lambda$, $|V|\le\lambda$, and $|M_i|\le 2^\lambda$ for each $i\in I$.
\item $\cu{N}'$ is special and  $|N'|$ is either finite or $2^\lambda$.
\item Every restricted conditional sentence that is true in $\cu{M}_i$ for most $i\in I$ is true in $\cu{N}'$.
\end{itemize}
Then there exists a filter $\cu{F}$ over $I$ such that $\cu{N}'\cong\prod_{\cu{F}}\cu{M}_i$.
\end{lemma}

\begin{proof}  The analogue of this lemma for classical model theory is Lemma 6.2.4 in [CK2012], and the reader can look at the proof in [CK2012] to fill in the details in this proof.  We first construct a mapping $h$ from $\prod_{i\in I} M_i$ onto $N'$ with the following property:
\medskip

 $\mathbf{P}$:
For every restricted conditional formula $\varphi(\vec{x})$ and tuple $\vec{a}$ in the domain of $h$, if $\varphi^{\cu{M}_i}(\vec{a}(i))=0$ for most $i\in I$, then $\varphi^{\cu{N}'}(h(\vec{a}))=0$.
\medskip

The hypothesis $2^\lambda=\lambda^+$  implies that $|\prod_{i\in I}M_i|\le \lambda^+$ and $|N'|\le\lambda^+$.  The mapping $h$ is constructed by enumerating both $\prod_{i\in I}M_i$
and $N'$ by sequences of length $\lambda^+$, and then using a back-and-forth argument that is a transfinite induction with $\lambda^+$ steps.
To begin we note that the hypotheses of this lemma guarantee that the empty mapping has property $\textbf{P}$.  At each stage of the induction, we assume that we have already constructed a mapping $h'$ from a subset of $\prod_{i\in I} M_i$ into $N'$ that has the  property $\mathbf{P}$.
Let $b$ be the next element in the enumeration of $\prod_{i\in I} M_i$, and let $e$ be the next element in the enumeration of $N'$.  We will use
the fact that the set of restricted conditional formulas has cardinality $\le\lambda$ and is closed under  $\max, \sup,\inf$, and that there are at most $\lambda$ previous stages.

We first use the fact that $\cu N'$ is $\lambda^+$-saturated to find an element $c\in N'$ such that $h'\cup\{(b,c)\}$ has the  property $\mathbf{P}$.
Let $\Gamma(y)$ be the set of all formulas of the form $\varphi(h'(\vec a),y)\dotle r$ such that $\varphi(\vec x,y)$ is restricted conditional,
$\vec a$ is a tuple in the domain of $h'$, $r$ is a positive dyadic rational, and $\varphi^{\cu M_i}(\vec a(i),b(i))=0$ for most $i\in I$.  For each finite subset $\Gamma_0\subseteq\Gamma$, the formula
$$\psi(\vec x):=(\inf_y\max(\Gamma_0))(\vec x,y)\dotle r$$
 is restricted conditional.  Moreover, $\psi^{\cu M_i}(\vec a(i))=0$ for most $i\in I$, so by property $\mathbf{P}$ for $h'$, $\psi^{\cu N'}(h'(\vec a))=0$.
It follows that $\Gamma(y)$ is finitely satisfiable in $\cu N'$.  The set $\Gamma(y)$ has cardinality $\le \lambda$, and $\cu N'$ is $\lambda^+$-saturated,
so $\Gamma(y)$ is satisfied by some element $c$ in $\cu N'$.  Then the mapping $h'\cup\{(b,c)\}$ has property $\mathbf{P}$.

We next use Lemma 6.1.6 in [CK2012] to find an element $d$ of $\prod_{i\in I} M_i$ such that $h'\cup\{(b,c),(d,e)\}$ has property $\mathbf{P}$.  That lemma says that for any set
$Z$ of cardinality $|Z|\le\lambda$ and family $X=\langle X_\zeta\colon \zeta\in Z\rangle$ of  sets of cardinality $\lambda$,
there is a family $Y=\langle Y_\zeta\colon \zeta\in Z\rangle$ of  pairwise disjoint sets of cardinality $\lambda$, called a \emph{refinement} of $X$, such that
$Y_\zeta\subseteq X_\zeta$ for each $\zeta\in Z$.

If $e$ is already in the range of $h'$, we take $d$ to be the first element such that $h'(d)=e$.  Otherwise, we argue as follows.  Let $Z$ be the set of  $\varphi(\vec a)$ such that $\varphi(\vec x)=\sup_y\psi(\vec x,y)$, where $\psi$ is restricted conditional,
$\vec a$ is a tuple of elements of
the domain of $h'\cup\{(b,c)\}$, and $\varphi^{\cu N'}(h(\vec a))>0$.  Then $\varphi(\vec x)$ is also restricted conditional.
For each $\varphi(\vec a)\in Z$, let $X_{\varphi(\vec a)}$
be the set of $i\in I$ such that $\varphi^{\cu M_i}(\vec a(i))>0$.  Since $h'\cup\{(b,c)\}$ has property $\mathbf{P}$, it follows that $X_\zeta$ has cardinality $\lambda$ for
each $\zeta\in Z$.  Therefore $X$ has a refinement $Y$.  Since each $Y_\zeta\subseteq X_\zeta$, we may  build an element $d\in\prod_{i\in I} M_i$ such that for each $\varphi(\vec a)\in Z$
and $i\in Y_{\varphi(\vec a)}$ we have $\psi^{\cu M_i}(\vec a(i),d(i))>0$.  Since each $Y_\zeta$ has cardinality $\lambda$, the function $h'\cup\{(b,c),(d,e)\}$ has property $\mathbf{P}$.
This completes the induction.

Now let $h$ be a mapping from $\prod_{i\in I} M_i$ onto $N'$ with property $\mathbf{P}$.
Let $\cu{E}$ be the set of all $J\subseteq I$ such that for some atomic formula $\alpha(\vec{x})$, some tuple $\vec{a}$ in $\prod_{i\in I} M_i$, and some $\varepsilon\in[0,1)$,
$$  \alpha^{\cu{N}'}(h(\vec{a}))<1-\varepsilon\mbox{ and } J=\{i\in I\colon \alpha^{\cu{M}_i}(\vec{a}(i))<1-\varepsilon\}.$$

\emph{Claim 1:}  $\cu{E}$ has the finite intersection property, and in fact, the intersection of any finite subset of $\cu{E}$ has cardinality $\lambda$.

\emph{Proof of Claim 1.}  Let $J_0,\ldots,J_n\in \cu{E}$ and $K=J_0\cap\cdots\cap J_n$.  For each $k\le n$, take an atomic formula $\alpha_k$, tuple
$\vec{a}_k$ in $\prod_{i\in I} M_i$, and a dyadic rational $\varepsilon_k\in[0,1)$ such that
$$ \alpha_k^{\cu{N}'}(h(\vec{a}_k))<1-\varepsilon_k\mbox{ and } J_k=\{i\in I\colon \alpha_k^{\cu{M}_i}(\vec{a}_k(i))<1-\varepsilon_k\}.$$
Let
$$\psi=\max(\alpha_0\dotplus \varepsilon_0,\ldots,\alpha_n\dotplus \varepsilon_n)$$
and $\vec{b}=(\vec{a}_0,\ldots,\vec{a}_n)$.  Then
\begin{equation}  \label{eq-E}
 \psi^{\cu{N}'}(h(\vec{b}))<1\mbox{ and }K=\{i\in I\colon\psi^{\cu{M}_i}(\vec{b}(i))<1\}.
 \end{equation}
Assume, towards a contradiction, that $|K|<\lambda$ and let $\varphi=1-\psi$.  $\varphi$ is equivalent to the restricted primitive conditional formula
$$\min(1-(\alpha_0\dotplus\varepsilon_0),\ldots,1-(\alpha_n\dotplus\varepsilon_n)).$$
For most $i\in I$ we have $i\notin K$, so $\psi^{\cu{M}_i}(\vec{b}(i))=1$ and hence $\varphi^{\cu{M}_i}(\vec{b}(i))=0.$  Then by property $\mathbf{P}$,
we have $\varphi^{\cu{N}'}(h(\vec{b}))=0.$  But this means that $\psi^{\cu{N}'}(h(\vec{b}))=1,$ contradicting (\ref{eq-E}).  Therefore $|K|=\lambda$ and Claim 1 is proved.

By Claim 1, $\cu{E}$ generates a proper filter $\cu{F}$ over $I$.

\emph{Claim 2:} For each atomic formula $\alpha(\vec{x})$ and tuple $\vec{a}$ in $\prod_{i\in I}M_i$,
$$\alpha^{\cu{N}'}(h(\vec{a}))=\limsup_{\cu{F}}\alpha^{\cu{M}_i}(\vec{a}(i)).$$

\emph{Proof of Claim 2.}  Let $r_N=\alpha^{\cu{N}'}(h(\vec{a}))$, and
$$r_M=\limsup_{\cu{F}}\alpha^{\cu{M}_i}(\vec{a}(i))=\inf_{J\in\cu{F}}\sup_{i\in J} \alpha^{\cu{M}_i}(\vec{a}(i)).$$
 For every $\varepsilon\in[0,1]$ such that $r_N<1-\varepsilon$, we have
$$J_\varepsilon :=\{i\in I\colon\alpha^{\cu{M}_i}(\vec{a}(i))<1-\varepsilon\}\in\cu{E}\subseteq\cu{F},$$
so
$$r_M\le\sup_{i\in J_\varepsilon} \alpha^{\cu{M}_i}(\vec{a}(i))\le 1-\varepsilon.$$
Since this holds for all $\varepsilon>0$, we have $r_M\le r_N$.

Now suppose $r_M<r_N$, and take $\delta\in(r_M,r_N)$.  Since $r_M<\delta$, there is a set $J\in\cu{F}$ such that $\sup_{i\in J}\alpha^{\cu{M}_i}(\vec{a}(i))\le \delta$. Hence
\begin{equation}  \label{eq-J}
(\forall i\in J)\alpha^{\cu{M}_i}(\vec{a}(i))\dotminus\delta=0.
\end{equation}
  There are sets $J_0,\ldots,J_n\in\cu{E}$ such that $J_0\cap\cdots\cap J_n\subseteq J$.  Take $\alpha_k, \vec{a}_k,$ and $\varepsilon_k$ such that
\begin{equation}  \label{eq-alpha-k}
 \alpha_k^{\cu{N}'}(h(\vec{a}_k))<1-\varepsilon_k\mbox{ and } J_k=\{i\in I\colon \alpha_k^{\cu{M}_i}(\vec{a}_k(i))<1-\varepsilon_k\}.
 \end{equation}
Since $J_0\cap\cdots\cap J_n\subseteq J$ we have
$$I=J\cup(I\setminus J_0)\cup\cdots\cup(I\setminus J_n).$$
For each $k\le n$,
$$J_k=\{i\in I\colon 1-(\alpha^{\cu{M}_i}(\vec{a}_k(i))\dotplus\varepsilon_k)>0\},$$
 so
\begin{equation}  \label{eq-Jk}
(\forall i\in I\setminus J_k)\ 1-(\alpha^{\cu{M}_i}(\vec{a}_k(i))\dotplus\varepsilon_k)=0.
\end{equation}

Let $\varphi$ be the restricted primitive conditional formula
$$\varphi=\min(\alpha\dotminus \delta,1-(\alpha_0\dotplus\varepsilon_0),\ldots,1-(\alpha_n\dotplus\varepsilon_n)),$$
and $\vec{b}=(\vec{a},\vec{a}_0,\ldots,\vec{a}_n)$.  By (\ref{eq-J}) and (\ref{eq-Jk}),
$$(\forall i\in I) \varphi^{\cu{M}_i}(\vec{b}(i))=0,$$
so by property $\mathbf{P}$,
$$\varphi^{\cu{N}'}(h(\vec{b}))=0.$$
Therefore either
$$r_N=\alpha^{\cu{N}'}(h(\vec{a}))\le\delta,$$
or
$$ \alpha^{\cu{N}'}_k(h(\vec{a}_k))\ge 1-\varepsilon_k \mbox{ for some } k\le n.$$
But this contradicts (\ref{eq-alpha-k}) and the fact that $\delta\in (r_M,r_N)$.  Therefore $r_M=r_N$, and Claim 2 is proved.

Now let $\cu{M}'$ be the pre-reduced product of $\langle\cu{M}_i\colon i\in I\rangle$ modulo $\cu{F}$.  Then  $\prod_{\cu{F}}\cu{M}_i$ is the reduction of $\cu{M}'$, and $h$ is a mapping from $M'$ onto $N'$.  By Claim 2 and Remark \ref{r-atomic-reduced-product}, $h$ is an embedding from $\cu{M}'$ onto $\cu{N}'$, so $\cu{M}'\cong \cu{N}'$.  Therefore $\cu{N}'\cong\prod_{\cu{F}}\cu{M}_i$, and the proof is complete.
\end{proof}

\begin{thm}  \label{t-main}  Assume that $\aleph_0+|V|\le\lambda$ and the GCH at $\lambda$ holds.  The following are equivalent.

\begin{itemize}
\item[(i)]  $T$ is $S$-equivalent to a set of conditional sentences.
\item[(ii)]  $T$ is $S$-equivalent to a set of restricted conditional sentences.
\item[(iii)] For any proper filter $\cu{F}$ over $I$, and indexed family $\langle\cu{M}_i\colon i\in I\rangle$ of general models of $S\cup T$, if the reduced product $\prod_{\cu{F}}\cu{M}_i$ is a general model of $S$ then it is a general model of $T$.
\end{itemize}

\end{thm}

\begin{proof}  (ii) $\Rightarrow$ (i) is trivial.
Even without the hypothesis that $2^\lambda=\lambda^+$, the implication (i) $\Rightarrow$ (iii) follows from Theorem \ref{t-conditional-preserve}.

(iii) $\Rightarrow$ (ii):  Assume (iii). Let $U$ be the set of all restricted conditional sentences $\varphi$ such that every general model of $S\cup T$ is a general model of $U$. Then every general model of $S\cup T$ is a general model of $S\cup U$.  We will show that every general model $\cu{N}$ of $S\cup U$ is a general model of $S\cup T$.  It will then follow that $S\cup T$ and $S\cup U$ have the same general models, and (i) follows.

By the GCH at $\lambda$, $2^\lambda$ is a special cardinal. By Facts \ref{f-Skolem} and \ref{f-special-exist}, we may assume that $\cu{N}$ is a special model such that $N$ is either finite or of cardinality $2^\lambda$.  For each conditional sentence $\varphi$ that is not true in $\cu{N}$, we have $\varphi\notin U$, so there exists a model of $S\cup T$ in which $\varphi$ is not true.   Let $I$ be a set of cardinality $\lambda$.  Since there are  $\aleph_0+|V|\le\lambda$ restricted sentences,  there is an indexed family $\langle\cu{M}_i\colon i\in I\rangle$ of general models of $S\cup T$ with $|M_i|\leq 2^\lambda$ such that every restricted conditional sentence that is true in $\cu{M}_i$ for most $i\in I$ is true $\cu{N}$.
By Lemma \ref{l-conditional-most}, there exists a filter $\cu{F}$ over $I$ such that $\cu{N}\cong\prod_{\cu{F}}\cu{M}_i$. Then by (iii), $\cu{N}$ is a model of $T$, and (ii) follows.
\end{proof}

As a consequence, we obtain the analogue of Theorem \ref{t-main} for metric structures.

\begin{cor}  \label{c-main-metric}  Assume that $\aleph_0+|V|\le\lambda$, $L$ is a metric signature over $V$,
 $S, T$ are metric theories with signature $L$, and the GCH at $\lambda$ holds.   Then the following are equivalent.

\begin{itemize}
\item[(i)]  $T$ is $S$-equivalent to a set of conditional sentences.
\item[(ii)]  $T$ is $S$-equivalent to a set of restricted conditional sentences.
\item[(iii)] For any proper filter $\cu{F}$ over $I$, and indexed family $\langle\cu{M}_i\colon i\in I\rangle$ of metric models of $S\cup T$ with signature $L$, if the reduced product $\prod_{\cu{F}}\cu{M}_i$ is a metric model of $S$ with signature $L$, then it is a metric model of $T$ with signature $L$.
\end{itemize}

\end{cor}

\begin{proof}  Since $S$ is a metric theory with signature $L$, $S\models\met(L)$, and similarly for $T$.
As in the proof of Theorem \ref{t-main}, the implication (ii) $\Rightarrow$ (i) is trivial, ant
the implication (i) $\Rightarrow$ (iii) follows from Theorem \ref{t-conditional-preserve} (without the hypothesis $2^\lambda=\lambda^+$).

We next assume (iii) and prove Condition (iii) of Theorem \ref{t-main}.  To distinguish pre-metric structures from their downgrades, we will use the full notation $(\cu N,L)$ for a pre-metric structure.
Let $\langle\cu N_i\colon i\in I\rangle$ be an indexed family of general models of $S\cup T$
such that $\prod_{\cu F}\cu N_i$ is a general model of $S$.  Since $S\models\met(L)$, each $(\cu N_i,L)$ is a  pre-metric model of $S$.
By Corollary \ref{c-metric-product}, $(\prod_{\cu F}\cu N_i,L)$ is a pre-metric structure, and hence a pre-metric model of $S$.
For each $i$, let $(\cu M_i,L)$ be the completion of $(\cu N_i,L)$. By Corollary \ref{c-product-completion}, $(\prod_{\cu F}\cu M_i,L)$ is the completion of
$(\prod_{\cu F}\cu N_i,L)$.
Therefore $\prod_{\cu F}\cu M_i\equiv\prod_{\cu F}\cu N_i$, so  $(\prod_{\cu F}\cu M_i,L)$ is a metric model of $S$.
By (iii) above, $(\prod_{\cu F}\cu M_i,L)$ is a metric model of $T$.  It follows that $\prod_{\cu F}\cu N_i$ is a
general model of $T$.  This proves Condition (iii) of  Theorem \ref{t-main}.

Finally by Theorem \ref{t-main}, Condition (ii)  of Theorem \ref{t-main}, which is the same as Condition (ii) above, holds.
\end{proof}

We have not been able to answer the following question.

\begin{question}  \label{q-eliminate-ch}
Can the conclusions of Theorem \ref{t-main} and Corollary \ref{c-main-metric} be proved in ZFC (without the continuum hypothesis)?
\end{question}

However, in the next section we will give an affirmative answer in the special case  where $S$ is the empty theory.

\section{Eliminating the Continuum Hypothesis}  \label{s-elim-CH}

We will prove the following preservation result in ZFC.

\begin{thm}  \label{t-main-simpler}
  The following are equivalent.

\begin{itemize}
\item[(i)]  $T$ is equivalent to a set of conditional sentences.
\item[(ii)]  $T$ is equivalent to a set of restricted conditional sentences.
\item[(iii)] $T$ is preserved under reduced products.
\end{itemize}
\end{thm}

\begin{rmk}  If we assume the GCH at some $\lambda\ge |T|+\aleph_0$, then
Theorem \ref{t-main-simpler} follows from Theorem \ref{t-main} when $S$ is the empty set of sentences.
\end{rmk}

We will now introduce a method that allows us to go back and forth between general $[0,1]$-valued structures and first-order structures.
This method will allow us to prove Theorem \ref{t-main-simpler} in ZFC by using the analogous preservation result in first order model theory
(Fact \ref{f-Horn} (ii) below).

\begin{df}  For a continuous vocabulary $V$,  let $V_\downarrow$ be the first order vocabulary with the same function and constant symbols as $V$ and with an $n$-ary predicate symbol $P_{\le r}$ for each $n$-ary predicate symbol $P\in V$ and rational $r\in[0,1)$.

For a reduced general structure $\cu M$ with vocabulary $V$, let $\cu M_\downarrow$ be the first order structure with vocabulary $V_\downarrow$, equality, and the same universe, functions, and constants as $\cu M$, and such that for each $P, r,$ and $\vec x\in M^n$, $\cu M_\downarrow\models P_{\le r}(\vec x)$ iff $\cu M\models P(\vec x)\dotle r$.

We say that a first-order structure $\cu K$ with vocabulary $V_\downarrow$ is \emph{increasing} if for every $P\in V$ and rational $r\le s$ in $[0,1)$, we have
$\cu K\models (\forall \vec x)[P_{\le r}(\vec x)\Rightarrow P_{\le s}(\vec x)]$.
\end{df}

Throughout this section, $\cu M, \cu M',\cu M_i$ will denote reduced general structures with vocabulary $V$, and $\cu K, \cu K', \cu K_i$ will denote increasing first-order structures
with vocabulary $V_\downarrow$.

\begin{df} Let $\cu K_\uparrow$ be the reduction of the general structure $\cu N$ with vocabulary $V$, and the same universe, functions, and constants as $\cu K$, such that for each $n$-ary $P\in V$ and $\vec x\in K^n$,
$P^{\cu N}(\vec x)=\inf\{s\colon \cu K\models P_{\le s}(\vec x)\}$.
\end{df}

\begin{lemma} \label{l-downup} $\cu M_\downarrow$ is increasing, and $\cu M=\cu M_{\downarrow\uparrow}$.
\end{lemma}

\begin{proof}  The lemma follows from the observation that $P^{\cu M}(\vec x)\le r$ if and only if for all rational $s\in(r,1)$ we have $P^{\cu M}(\vec x)\le s$ .
\end{proof}

\begin{df}
For a continuous theory $T$ with vocabulary $V$, let $T_\downarrow$ be the theory of the class of all increasing $V_\downarrow$-structures $\cu K$
such that $\cu K_\uparrow\models T$, that is,
$$T_\downarrow=\{\theta\colon (\forall \cu K)[\cu K_\uparrow\models T\Rightarrow \cu K\models \theta ]\}.$$
\end{df}

\begin{lemma}  Let $T$ and $U$ be continuous theories.
\begin{itemize}
\item[(i)]  If $T\models U$ then $T_\downarrow\models U_\downarrow$.
\item[(ii)] If $\cu M\models T$, then $\cu M_\downarrow\models T_\downarrow$.
\item[(iii)]  Every model of $T_\downarrow$ is increasing.
\end{itemize}
\end{lemma}

\begin{proof}  (i) is clear from the definitions.

(ii) Suppose $\cu M\models T$ and $\theta\in T_\downarrow$.  By Lemma \ref{l-downup}, $\cu M=(\cu M_\downarrow)_\uparrow\models T$, so $\cu M_\downarrow\models\theta$.
Therefore $\cu M_\downarrow\models T_\downarrow$.

(iii) The property of being increasing is expressed by a set of $V_\downarrow$-sentences.
\end{proof}

The following key lemma shows that one can freely move the $\uparrow$ symbol in certain reduced product formulas.

\begin{lemma} \label{p-down-reduced-product}
Let $\cu F$ be a proper filter over $I$.  Then
$$\prod_{\cu F} (\cu K_{i\uparrow}) = \left(\prod_{\cu F}(\cu K_i)\right)_\uparrow.$$  In particular, $$\prod_{\cu F}\cu{M}_i=\left(\prod_{\cu F}(\cu{M}_i)_\downarrow\right)_\uparrow.$$

\end{lemma}

\begin{proof}  It is enough to prove the result in the case that $V$ contains a predicate symbol $d$ for the discrete metric,
because then the general case follows by removing $d$ from the vocabulary and taking the reduction of both sides.
In that case, for each $n$-ary $P\in V$, rational $r\in[0,1)$, and $\vec x\in (\prod K_i)^n$, the following are equivalent.

\begin{itemize}
\item $\prod_{\cu F} (\cu K_{i\uparrow})\models P(\vec x_\cu F)\dotle r$.
\item $\inf_{J\in\cu F}(\sup_{i\in J}P^{\cu K_{i\uparrow}}(\vec x_i))\dotle r$.
\item $(\forall s\in(r,1))(\exists J\in\cu F)(\sup_{i\in J}P^{\cu K_{i\uparrow}}(\vec x_i))\dotle s$.
\item $(\forall s\in(r,1))(\exists J\in\cu F)(\forall i\in J)P^{\cu K_{i\uparrow}}(\vec x_i)\dotle s$.
\item $(\forall s\in(r,1))(\exists J\in\cu F)(\forall i\in J)\cu K_i\models P_{\le s}(\vec x_i)$.
\item $(\forall s\in(r,1))\{i\colon \cu K_i\models P_{\le s}(\vec x_i)\}\in\cu F$.
\item $(\forall s\in(r,1))\prod_{\cu F}(\cu K_i)\models P_{\le s}(\vec x_{\cu F})$.
\item $(\prod_{\cu F}(\cu K_i))_\uparrow\models P(\vec x_{\cu F})\dotle r$.

\end{itemize}
The second statement follows from the first statement and Lemma \ref{l-downup}.
\end{proof}

\begin{lemma} \label{l-uparrow-equiv}

(i) If $\cu K\cong\cu K'$ then $\cu K_{\uparrow}\cong\cu K'_{\uparrow}$.

(ii) If $\cu K\equiv\cu K'$ then $\cu K_{\uparrow}\equiv\cu K'_{\uparrow}$.
\end{lemma}

\begin{proof}   (i) is clear.  We prove (ii).  By the Isomorphism Theorem for ultrapowers (Theorem 6.1.15 in [CK2012]),
there are ultrafilters $\cu D$ over $H$ and $\cu E$ over $I$ such that $\cu K^{H}/{\cu D}\cong\cu K'^{J}/{\cu E}$.
Then by (i) we have $(\cu K^{H}/{\cu D})_\uparrow\cong(\cu K'^{J}/{\cu E})_\uparrow$.
By Lemma \ref{p-down-reduced-product}, $(\cu K_{\uparrow})^{H}/{\cu D}\cong(\cu K'_{\uparrow})^{J}/{\cu E}$.
It follows that $\cu K_{\uparrow}\equiv\cu K'_{\uparrow}$.
\end{proof}

\begin{lemma}  \label{p-down-equiv}  For each continuous theory $ T$ and each $\cu K$, $\cu K\models T_\downarrow$ if and only if $\cu K_\uparrow\models T$.
\end{lemma}

\begin{proof}
It follows immediately from the definition of $T_\downarrow$ that if $\cu K_\uparrow\models T$, then $\cu K\models T_\downarrow$.  For the other direction, suppose that $\cu K\models T_\downarrow$.  Then there is a set $I$, an ultrafilter $\cu D$ on $I$, and increasing structures $\cu K_i$, $i\in I$, such that $\cu K_{i\uparrow} \models T$ for each $i\in I$, and $\cu K\equiv \prod_{\cu D}\cu K_i$.  By Lemmas \ref{l-uparrow-equiv} and \ref{p-down-reduced-product}, we have $$\cu K_\uparrow \equiv \left(\prod_{\cu D}\cu K_i\right)_\uparrow=\prod_{\cu D}(\cu K_{i\uparrow}),$$
from which it follows that $\cu K_\uparrow \models T$.
\end{proof}

\begin{lemma}  \label{p-preserved-iff}
For all continuous theories $S, T$, the following are equivalent.

\noindent\begin{itemize}
\item[(i)] Each reduced product of general models of $S$ is a general model of $ T$.
\item[(ii)]  Each reduced product of  models of $S_\downarrow$ is a  model of $ T_\downarrow$.
\end{itemize}
\end{lemma}

\begin{proof}  Let $\cu F$ be a proper filter over a set $I$.

(i) $\Rightarrow$ (ii): Assume (i).  For each $i\in I$, let $\cu K_i\models S_\downarrow$.
By Lemma \ref{p-down-equiv}, $\cu K_{i\uparrow}\models S$, so $\prod_{\cu F}(\cu K_{i\uparrow})\models T$.
But $\prod_{\cu K}(\cu K_{i\uparrow})=(\prod_{\cu F}(\cu K_i))_\uparrow$ by Lemma \ref{p-down-reduced-product}.
Therefore  $(\prod_{\cu F}(\cu K_i))_\uparrow\models T$, so $\prod_{\cu F}(\cu K_i)\models T_\downarrow$ by Lemma \ref{p-down-equiv}.
This proves (ii).

(ii) $\Rightarrow$ (i): Assume (ii). Let $\cu M_i\models S$ for each $i\in I$.
 We have $\cu M_{i\downarrow\uparrow}=\cu M_i$ by Lemma \ref{l-downup}, so $\cu M_{i\downarrow}\models S_\downarrow$.
Then $\prod_{\cu F}(\cu M_{i\downarrow})\models T_\downarrow$.  By Lemma \ref{p-down-equiv}, $(\prod_{\cu F}(\cu M_{i\downarrow}))_\uparrow\models T$.
By the second statement of Lemma \ref{p-down-reduced-product}, we have $ \prod_{\cu F}(\cu M_i)\models T$, and (i) holds.
\end{proof}

\begin{cor}  \label{c-preserved-iff}
For every continuous theory $T$, $T$ is preserved under reduced products if and only if $T_\downarrow$ is
preserved under reduced products.
\end{cor}

\begin{proof}  Take $S=T$ in Lemma \ref{p-preserved-iff}.
\end{proof}

In order to eliminate the continuum hypothesis from Theorem \ref{t-main}, we will need to use the Shoenfield Absoluteness Theorem.  For background we refer to [J].  The set of \emph{hereditarily finite
sets}, which is countable, is denoted by $\HF$.  Note that tuples of hereditarily finite sets are also hereditarily finite sets.
 A $\Pi^1_2$ \emph{sentence (over $\HF$)} is a sentence of set theory of the form
$$(\forall X\subseteq \HF)(\exists Y\subseteq \HF)\theta(X,Y)$$
where $\theta(X,Y)$ is a formula of set theory with quantifiers restricted to $\HF$.

\begin{fact}  \label{f-absorb}
Every sentence of set theory of the form $Q_1 Q_2\theta$, where $Q_1$ is a sequence of second order universal quantifiers over subsets of $\HF$ and first order
quantifiers restricted to $\HF$,  $Q_2$ is a sequence of second order existential quantifiers over subsets of $\HF$ and first order
quantifiers restricted to $\HF$, and $\theta$ is quantifier-free, is equivalent to a $\Pi^1_2$ sentence.
\end{fact}

We will use the following consequence of the Shoenfield Absoluteness Theorem.

\begin{fact}  \label{f-Shoenfield}  Every $\Pi^1_2$ sentence over $\HF$ that is provable from ZFC + CH is provable from ZFC.
\end{fact}

If the vocabulary $V$ is a subset of $\HF$, then by coding formulas in the usual way, we can also take the set of restricted continuous
formulas with vocabulary $V$, and the set of first order formulas in the vocabulary $V_\downarrow$, to be subsets of $\HF$.

\begin{lemma}  \label{l-Horn-conditional-iff}  (In ZFC)
Suppose $V\subseteq\HF$, $T$ is a set of restricted continuous sentences in vocabulary $V$, and $T_\downarrow$ is equivalent to a  set of Horn sentences.  Then
 $T$ is equivalent to a  set of restricted conditional sentences.
\end{lemma}

\begin{proof}
First assume the continuum hypothesis.  By Fact \ref{f-Horn}, $T_\downarrow$ is preserved under reduced products.
By Corollary \ref{c-preserved-iff}, $T$ is preserved under reduced products.  Since $V\subseteq \HF$, $V$ is countable.
By the continuum hypothesis and Theorem \ref{t-main}, $T$ is equivalent to a  set of restricted conditional sentences.

By Fact \ref{f-Shoenfield}, to show that this lemma is provable in ZFC, it suffices to show that
the statement of this lemma is  a $\Pi^1_2$ sentence. To do that,
we will freely use Fact \ref{f-absorb}, which allows us to ignore quantifiers over elements of $\HF$.

It is elementary that the statements ``$\theta$ is a first order $V_\downarrow$-sentence'', ``$\theta$ is a Horn sentence'',  `` $\varphi$ is a strict $V$-sentence'',
and ``$\varphi$ is a strict conditional sentence'', are $\Delta^1_1$.  It follows that the first hypothesis of this lemma, ``$T$ is a set of strict continuous sentences in
a vocabulary $V\subseteq \HF$'', is $\Delta^1_1$.

Every countable $V_\downarrow$-structure is isomorphic to a $V_\downarrow$-structure that is a subset of $\HF$ and similarly for general structures.
In the following we let $\cu K$ be a variable that ranges over $V_\downarrow$-structures that are subsets of $\HF$ with universe $K$,
and let $\cu M, \cu N$ be variables ranging over $V$-structures that are subsets of $\HF$.
By unravelling the definition of satisfaction in [CK2012] and [BBHU], one can show that the statements
``$\cu K\models\theta$'', ``$\cu K$ is increasing'', ``$\cu K_\uparrow\models\varphi$'', ``$\cu K_\uparrow\models T$'', ``$\cu N\models\varphi$'', and ``$\cu N\models T$'', are $\Delta^1_1$.

To illustrate, we show that the statement ``$\cu K_\uparrow\models \varphi$'' is $\Delta^1_1$.   The statement ``$\psi$ is a strict $V$-sentence with parameters in $\cu K$ and $r$ is a dyadic rational in $[0,1)$'' is a $\Delta^1_1$ formula $\delta(\cu K,\psi,r)$.  Let us say that $D$ is an $\uparrow$-\emph{diagram} of $\cu K$ if $D$ is the set of all pairs $(\psi,r)$ such that $\delta(\cu K,\psi,r)$ and $\cu K_\uparrow\models \psi\dotle r$. Then $D$ is an $\uparrow$-diagram of $\cu K$ if and only if each of the following $\Delta^1_1$ sentences with the parameters $\cu K, D$ hold.
\begin{itemize}
\item $(\forall \psi)(\forall r)[(\psi,r)\in D\Rightarrow \delta(\cu K,\psi,r)] .$
\item $(\forall \mbox{ atomic }P(\vec\tau))(\forall r)[(P(\vec \tau),r)\in D\Leftrightarrow\cu K \models P_{\le r}( \vec \tau)]$.
\item $(\forall \psi_1)(\forall \psi_2)(\forall r)[(\max(\psi_1,\psi_2),r)\in D\Leftrightarrow [(\psi_1,r)\in D\wedge(\psi_2,r)\in D]].$
\item Similar rules for connectives $\min, \dotminus,\dotplus, \cdot/2$.
\item $(0,r)\in D\wedge (1,r)\notin D$.
\item $(\forall \psi(x))(\forall r)[(\sup_x\psi(x),r)\in D\Leftrightarrow (\forall a\in K)(\psi(a),r)\in D]$.
\item A similar rule for $\inf_x$.
\end{itemize}

Then $\cu K_\uparrow\models\varphi$ if and only if $(\varphi,0)\in D$ for every $\uparrow$-diagram $D$ of $\cu K$, and also if and only if
$(\varphi,0)\in D$ for some $\uparrow$-diagram $D$ of $\cu K$.  This shows that the statement ``$\cu K_\uparrow\models\varphi$'' is $\Delta^1_1$.

By Lemma \ref{p-down-equiv}, the statement ``$\cu K\models T_\downarrow$'' is also $\Delta^1_1$.
It follows that the second hypothesis of this lemma, that ``$T_\downarrow$ is equivalent to a set of Horn sentences'',
is equivalent in ZFC to the following $\Sigma^1_2$ sentence with the parameter $T$:
$$(\exists U\subseteq\HF)(\forall\cu K)[(\forall \theta)(\theta\in U\Rightarrow \theta \mbox{ is Horn}) \wedge (\cu K\models T_\downarrow \Leftrightarrow \cu K\models U)].$$
The statement ``$T\models\varphi$'' is $\Pi^1_1$, because it is equivalent in ZFC to
$$ (\forall \cu N)[\cu N\models T\Rightarrow \cu N\models\varphi].$$
The conclusion of this lemma says that
for every  $\cu M$, if every strict conditional consequence of $T$ holds in $\cu M$, then $\cu M\models T$, that is,
$$(\forall\cu M)[(\forall\varphi)(\varphi \mbox{ strict conditional and } T \models\varphi)\Rightarrow\cu M\models T].$$
This is  a $\Pi^1_2$ sentence with a parameter for $T$.

So the whole statement of this lemma has the form
$$(\forall T)[\alpha(T)\Rightarrow \beta(T)]$$
where $\alpha(T)$ is  a $\Sigma^1_2$ sentence, and $\beta(T)$ is  a $\Pi^1_2$ sentence.  It follows that the  statement of this lemma is a $\Pi^1_2$ sentence.

\end{proof}

We are now ready to prove Theorem \ref{t-main-simpler} in ZFC.  For convenience we restate the result here.

\setcounter{thm}{0}

\begin{thm}  (Restated)
  The following are equivalent.

\begin{itemize}
\item[(i)]  $T$ is equivalent to a set of conditional sentences.
\item[(ii)]  $T$ is equivalent to a set of restricted conditional sentences.
\item[(iii)] $T$ is preserved under reduced products.
\end{itemize}
\end{thm}

\begin{proof}
  (ii) $\Rightarrow$ (i) is trivial, and (i) $\Rightarrow$ (iii) follows from Theorem  \ref{t-conditional-preserve}

(iii) $\Rightarrow$ (ii):  Assume (iii).  By Lemma \ref{l-restricted-full}, we may assume without generality that $T$ is a set of
strict continuous sentences, and  that every consequence of $T$ that is a strict continuous sentence in vocabulary $V$ belongs to $T$.
For each vocabulary $V'\subseteq V$, let $[V']$ be the set of strict continuous sentences in vocabulary $V'$, and let $T'=T\cap[V']$.

\medskip

\textbf{Claim 1.}  For every countable vocabulary $V'\subseteq V$, $T'$ is preserved under reduced products.
\medskip

\emph{Proof of Claim 1}.
 Let $\{ \cu M'_i\colon i\in I\}$ be a family of general models of $T\cap[V']$,  let $\cu F$ be a proper filter over $I$,
and let $\cu M'=\prod_{\cu F}(\cu M'_i)$.
For each $i$, let $U_i$ be the set of strict continuous sentences
$$U_i=T\cup\{\varphi\dotle r\colon \varphi\in Th(\cu M'_i), r>0\}.$$
Suppose that $U_i$ is not finitely satisfiable. Then there is a strict sentence $\varphi\in Th(\cu M'_i)$ and a dyadic rational $r>0$ such that $T\models r\dotle\varphi$.
The sentence $r\dotle\varphi$ belongs to $[V']$, and also belongs to $T$ because it is a consequence of $T$.
But then $\cu M'_i\models\varphi = 0$ and $\cu M'_i\models r\dotle \varphi$, which is a contradiction.  Therefore $U_i$ is finitely satisfiable.
By the Compactness Theorem, $U_i$ has a general model $\cu N_i$.  Then $\cu N_i\models T$ for each $i$, so by (iii) we have $\cu N:=\prod_{\cu F}(\cu N_i)\models T$.
Let $\cu N'_i$ be the $V'$-part of $\cu N_i$.  By Remark \ref{r-part-reduced-product}, $\cu N':=\prod_{\cu F}(\cu  N'_i)$ is the $V'$-part of $\cu N$, so $\cu N'\models T'$.
Since $\cu N_i\models U_i$, we have $\cu N'_i\equiv \cu M'_i$.  Then
 by Proposition \ref{p-Gh-general}, $\cu N'\equiv\cu M'$.  Therefore $\cu M'\models T'$, and Claim 1 is proved.
\medskip

\textbf{Claim 2.}  For every countable $V'\subseteq V$, $T'$ is equivalent to a set of strict conditional sentences.
\medskip

\emph{Proof of Claim 2.} By Claim 1, $T'$ is preserved under reduced products.
We may take $V'$, $T'$, $V'_\downarrow$, and $T'_\downarrow$ to be subsets of $\HF$.
 By Lemma \ref{p-preserved-iff}, $T'_\downarrow$ is preserved under reduced products.  By Fact \ref{f-Horn} (ii),
 $T'_\downarrow$ is equivalent to a set of Horn sentences with vocabulary $V_\downarrow$.  By Lemma \ref{l-Horn-conditional-iff},
 $T'$ is equivalent to a set of strict conditional sentences in the vocabulary $V'$, so Claim 2 is proved.

It follows at once from Claim 2 that Condition (i) holds.
\end{proof}

\setcounter{thm}{16}

Here is a version of Theorem \ref{t-main-simpler} for single sentences.

\begin{cor} \label{c-main-single}
For each continuous sentence $\varphi$, the following are equivalent.

\begin{itemize}

\item[(i)]  For each positive integer $n$ there is a conditional sentence $\psi_n$ such that $\varphi\models\psi_n$ and $\psi_n\models\varphi\dotle 2^{-n}$.
\item[(ii)]  For each positive integer $n$ there is a restricted conditional sentence $\psi_n$ such that $\varphi\models\psi_n$ and $\psi_n\models\varphi\dotle 2^{-n}$.
\item[(iii)]  $\varphi$ is preserved under reduced products.
\end{itemize}
\end{cor}

\begin{proof}  It is clear that (ii) implies (i). 

 Assume (i).  Then $\{\varphi\}$ is equivalent to the countable set $\{\psi_n\colon n>0\}$ of conditional sentences.
Then by Theorem \ref{t-main-simpler}, (iii) holds.

Now assume (iii).  By Theorem \ref{t-main-simpler}, $\{\varphi\}$ is equivalent to a set $U$  of restricted conditional sentences.  By Corollary \ref{c-compactness-approx},
for each $n>0$ there is a finite $U_0\subseteq U$ such that $U_0\models \varphi\dotle 2^{-n}$.  Let $\psi_n=\max(U_0)$.  Then $\psi_n$ is
a restricted conditional sentence,  $\varphi\models\psi_n$ and $\psi_n\models \varphi\dotle 2^{-n}$, so (ii) holds.
\end{proof}

The following corollary characterizes metric theories that are preserved under reduced products of metric structures.

\begin{cor}  \label{c-sentence-main-metric}
 Let $L$ be a metric signature over $V$, and let $T$ be a metric theory with signature $L$.  The following are equivalent:

\begin{itemize}
\item[(i)]  $T$ is equivalent to a set  of conditional sentences.
\item[(ii)]  Every  reduced product of pre-metric models of $T$is a pre-metric model of $T$ (with signature $L$).
\item[(iii)]  Every  reduced product of metric models of $T$ is a metric model of $T$ (with signature $L$).
\end{itemize}
\end{cor}

\begin{proof}  (i) $\Rightarrow$ (ii) follows from Theorem \ref{t-conditional-preserve} and Corollary \ref{c-metric-product}.  (ii) $\Rightarrow$ (iii) follows from
Corollary \ref{c-metric-product}.

(iii) $\Rightarrow$ (ii):  Assume (iii), and let $(\prod_{\cu F}\cu N_i,L)$ be a reduced product of pre-metric models of $T$.
For each $i\in I$, let $(\cu M_i,L)$ be the completion of $(\cu N_i,L)$, which is a metric model of $T$.  By  (iii) above,
$(\prod_{\cu F} \cu M_i,L)$ is a metric model of $T$.  By Corollary \ref{c-product-completion}, $(\prod_{\cu F} \cu M_i,L)$ is the completion of $(\prod_{\cu F} \cu N_i,L)$.
Therefore $(\prod_{\cu F} \cu N_i,L)$ is a pre-metric model of $T$, and (ii) holds.

(ii) $\Rightarrow$ (i):  Assume (ii).  We first show that $T$ is preserved under reduced products of general structures.  Suppose $\prod_{\cu F} \cu N_i$ is a reduced product of
general models of $T$.  Since $T\models\met(L)$, $(\cu N_i,L)$ is a pre-metric model of $T$ for each $i\in I$.  Therefore, by (ii), $(\prod_{\cu F}\cu N_i,L)$
is a pre-metric model of $T$, and hence $\prod_{\cu F}\cu N_i$ is a general model of $T$.  Thus $T$ is preserved under reduced products of general structures.
So by Theorem \ref{t-main-simpler}, (i) holds.
\end{proof}

In [Ga], Galvin also proved the following ``interpolation'' statement:  For all first order theories $S_0, T_0$, if every reduced product of models of $S_0$ is a model of $T_0$, then there is a set $U_0$ of Horn sentences
such that $S_0\models U_0$ and $U_0\models T_0$.  It is thus natural to ask if the continuous analogue of this fact is true:

\begin{question}  \label{q-conditional-interpolaate-ZFC}
Suppose $S, T$ are  continuous theories, and every reduced product of general models of $ S$  is a general model of $T$.
Is there a set $U$ of conditional sentences such that $S\models U$ and $U\models T$?
\end{question}

We end this section with one further characterization of reduced product sentences.  First, we need a lemma.  For our purposes, by a \emph{definable predicate}, we mean an expression of the form $\psi(\vec x):=\sum_n 2^{-n}\psi_n(\vec x)$, where each $\psi_n(\vec x)$ is a formula.

\begin{lemma}\label{l-implication}
For definable predicates $\psi(\vec x)$ and $\chi(\vec x)$, the following are equivalent:
\begin{enumerate}
\item For all general structures $\cu{M}$ and all $\vec a$ from $M$, if $\cu{M}\models \psi(\vec a)$, then $\cu{M}\models \chi(\vec a)$.
\item There is an increasing continuous function $\alpha:[0,1]\to[0,1]$ such that $\alpha(0)=0$ and for which, given any general structure $\cu{M}$ and $\vec a$ from $M$, we have $\chi(\vec a)^{\cu{M}}\leq \alpha(\psi(\vec a)^{\cu{M}})$.
\end{enumerate}
\end{lemma}

\begin{proof}
In the case of metric structures, this is Proposition 7.15 of [BBHU].  However, the proof given there also works in the case of general structures.
\end{proof}

\begin{cor}
Given a sentence $\varphi$, the following are equivalent:
\begin{enumerate}
\item $\varphi$ is a reduced product sentence.
\item There is an increasing continuous function $\gamma:[0,1]\to [0,1]$ such that $\gamma(0)=0$ and for which, given any set $I$, general structures $\cu{M}_i$ for each $i\in I$, and a filter $\cu{F}$ on $I$, setting $\cu{M}:=\prod_{\cu{F}} \cu{M}_i$, we have
$$\varphi^{\cu{M}}\leq \gamma\left(\limsup_{\cu{F}}\varphi^{\cu{M}_i}\right).$$
\end{enumerate}
\end{cor}

\begin{proof}
It is clear that (2) implies (1).  Conversely, suppose that (1) holds.  Then by Corollary \ref{c-main-single}, there are conditional sentences $\psi_n$ such that $\sigma\models \psi_n$ and $\psi_n\models \sigma\dotle 2^{-n}$.  Set $\psi:=\sum_n 2^{-n}\psi_n$, a definable predicate.  Note that $\varphi\models \psi$ and $\psi\models \varphi$.  By Lemma \ref{l-implication}, there are increasing continuous functions $\alpha_n:[0,1]\to [0,1]$ and $\beta:[0,1]\to[0,1]$ with $\alpha_n(0)=\beta(0)=0$ and such that $\psi_n^{\cu{M}}\leq \alpha_n(\varphi^{\cu{M}})$ and $\varphi^{\cu{M}}\leq \beta(\psi^{\cu{M}})$ for all general structures $\cu{M}$.  Set $\alpha:=\sum_n 2^{-n}\alpha_n$, so that $\alpha:[0,1]\to [0,1]$ is an increasing continuous function with $\alpha(0)=0$.  Finally, set $\gamma:=\beta\circ \alpha$.

Suppose now that we have a set $I$, general structures $\cu{M}_i$ for each $i\in I$, and a filter $\cu{F}$ on $I$, and set $\cu{M}:=\prod_{\cu{F}} \cu{M}_i$.  Let $r:=\limsup_{\cu{F}} \varphi^{\cu{M}_i}$ and take $s>r$.  Take $J\in \cu{F}$ such that $\sup_{i\in J}\varphi^{\cu{M}_i}\leq s$.  It follows that $\psi_n^{\cu{M}_i}\leq \alpha_n(s)$ for all $i\in J$.  By (4) in the proof of Theorem \ref{t-conditional-preserve}, it follows that $\psi_n^{\cu{M}}\leq \alpha_n(s)$ and thus $\psi^{\cu{M}}\leq \alpha(s)$ and $\varphi^{\cu{M}}\leq \beta(\psi^{\cu{M}})\leq\beta(\alpha(s))=\gamma(s)$.  Since $s>r$ was arbitrary, we have that $\varphi^{\cu{M}}\leq \gamma(r)$, as desired.
\end{proof}

\appendix

\section{The Feferman-Vaught Theorem for General Structures}

In [Gh], Ghasemi proved an analogue of the Feferman-Vaught Theorem for reduced products of metric structures.  He used that result to prove that reduced products of metric structures preserve elementarily equivalence (Fact \ref{f-Gh} above), which we in turn used  to prove that reduced products of general structures preserve elementary equivalence.  In this appendix, we will extend Ghasemi's  analogue of the Feferman-Vaught Theorem to general structures.

In the following definition, we will slightly strengthen the notion from [Gh] of a formula being determined up to $2^{-n}$, by adding the additional requirement (c).
Let $L$ be a metric signature.

\begin{df} \label{d-determined} For a restricted continuous formula $\varphi(\vec x)$ with signature $L$, we say that $\varphi(\vec x)$ is \emph{determined up to} $2^{-n}$ by
$$(\sigma_0,\ldots, \sigma_{2^n};\psi_0,\ldots,\psi_{m-1})$$
if
\begin{itemize}
\item[(a)] Each $\sigma_i$ is a formula in the first order language of Boolean algebras with at most $s=m2^n$ variables which is monotonic in each variable.
\item[(b)] Each $\psi_j(\vec x)$ is a restricted continuous formula.
\item[(c)] For each $j$, every predicate or function symbol that occurs in $\psi_j$ occurs in $\varphi$.
\item[(d)] For any set $\Omega$, ideal $\cu{I}$ and corresponding filter $\cu F$ on $\Omega$, indexed family $\langle\cu M_\gamma\rangle_{\gamma\in\Omega}$
of metric structures with signature $L$, $|\vec x|$-tuple $\vec a$ in $\prod_\Omega \cu M_\gamma$, and $\ell\in\{0,\ldots,2^n\}$, we have
$$\cu{P}(\Omega)/\cu{I}\models\sigma_\ell([X^0_0]_{\cu I},\ldots, [X^0_{2^n}]_{\cu I},\ldots,[X^{m-1}_{2_n}]_{\cu I})
\Rightarrow \varphi(\vec a_{\cu F})^{\prod_{\cu F}\cu M_\gamma} > \ell/2^n,$$
and
$$ \varphi(\vec a_{\cu F})^{\prod_{\cu F}\cu M_\gamma} > \ell/2^n\Rightarrow
\cu{P}(\Omega)/\cu{I}\models\sigma_\ell([Y^0_0]_{\cu I},\ldots, [Y^0_{2^n}]_{\cu I},\ldots,[Y^{m-1}_{2^n}]_{\cu I}),$$
where for each $i$ and $j$,
$$X^j_i=\{\gamma\in\Omega\colon \psi_j(\vec a(\gamma))^{\cu M_\gamma} > i/2^n\}, Y^j_i=\{\gamma\in\Omega\colon \psi_j(\vec a(\gamma))^{\cu M_\gamma} \ge i/2^n\}.$$
\end{itemize}
\end{df}

\begin{df}  We say that a formula $\varphi$ with signature $L$ is \emph{determined up to} $2^{-n}$ if there is a restricted formula $\theta$ such that every predicate or
function symbol that occurs in $\theta$ occurs in $\varphi$, $\theta$ is uniformly within $2^{-(n+1)}$ of $\varphi$ in all metric structures, and $\theta$ is determined
up to $2^{-n}$ by some $(\sigma_0,\ldots,\sigma_{2^n},\psi_0,\ldots,\psi_{m-1})$.
\end{df}

\begin{fact}  \label{f-FV} (Theorem 3.3 of [Gh]) For every $n\in\BN$, every continuous formula is determined up to $2^{-n}$.
\end{fact}

The above statement is slightly stronger than the statement of Theorem 3.3 in [Gh] because of condition (c) in our definition of being determined, but the proof in [Gh]
shows that the result holds as stated here.

Now consider a vocabulary $V$ with countably many predicate and function symbols.

\begin{df}  For a restricted continuous formula $\varphi(\vec x)$ with vocabulary $V$, we say that $\varphi(\vec x)$ is \emph{generally determined up to} $2^{-n}$ by
$$(\sigma_0,\ldots, \sigma_{2^n};\psi_0,\ldots,\psi_{m-1})$$
if conditions (a)--(d) of Definition \ref{d-determined} hold with the phrase ``metric structures with signature $L$''
replaced by ``general structures with vocabulary $V$'' in (d).

We say that a continuous formula $\varphi$ with vocabulary $V$ is \emph{generally determined up to} $2^{-n}$ if there is a restricted formula $\theta$ such that every predicate or
function symbol that occurs in $\theta$ occurs in $\varphi$, $\theta$ is uniformly within $2^{-(n+1)}$ of $\varphi$ in all general structures, and $\theta$ is generally determined
up to $2^{-n}$ by some $(\sigma_0,\ldots,\sigma_{2^n},\psi_0,\ldots,\psi_{m-1})$.
\end{df}

\begin{thm}  \label{t-determined}  Every continuous formula $\varphi$ with vocabulary $V$ is generally determined.
\end{thm}

\begin{proof}  By the Expansion Theorem, the empty theory $T$ with vocabulary $V$ has a pre-metric expansion $(T_e, L_e)$.  Fix an $n\in\BN$.
By Fact \ref{f-FV}, $\varphi$ is determined up to $2^{-n}$ with respect to the metric signature $L_e$.  Then there is a restricted formula
$\theta$ such that every predicate or
function symbol that occurs in $\theta$ occurs in $\varphi$, $\theta$ is uniformly within $2^{-(n+1)}$ of $\varphi$ in all metric structures, and $\theta$ is determined
up to $2^{-n}$ by some $(\sigma_0,\ldots,\sigma_{2^n},\psi_0,\ldots,\psi_{m-1})$.  Then $\theta$  and each $\psi_i$ are $V$-formulas (rather than just $V_D$-formulas).
Then $\theta$ is uniformly within $2^{-(n+1)}$ of $\varphi$ in all general $V$-structures.  Let $\langle\cu N_\gamma\rangle_{\gamma\in\Omega}$ be an indexed family of general $V$-structures.
For each $\gamma\in\Omega$, $\cu N_{\gamma e}$ is a pre-metric structure with signature $L_e$.  Then the completion $\cu M_{\gamma e}$ of $\cu N_{\gamma e}$
is a metric structure with signature $L_e$.  By Corollary \ref{c-product-completion}, for each filter $\cu F$ over $\Omega$, $\prod_{\cu F}\cu M_{\gamma e}$
is the completion of $\prod_{\cu F}\cu N_{\gamma e}$.  Since $\theta$ is determined
up to $2^{-n}$ by some $(\sigma_0,\ldots,\sigma_{2^n},\psi_0,\ldots,\psi_{m-1})$, Condition (d) of Definition \ref{d-determined} holds for $\theta$ and
$\langle\cu M_\gamma\rangle_{\gamma\in\Omega}$.  It follows that Condition (d) of Definition \ref{d-determined} also holds for $\theta$ and
$\langle\cu N_\gamma\rangle_{\gamma\in\Omega}$. Therefore $\theta$ is generally determined up to $2^{-n}$ by some $(\sigma_0,\ldots,\sigma_{2^n},\psi_0,\ldots,\psi_{m-1})$,
so $\varphi$ is generally determined.
\end{proof}

\section{Embeddings, Unions of Chains, and Homomorphisms}

In this appendix, we prove some relatively easy preservation theorems for general structures.  Each of these results is a special case of
a  result  in the much earlier monograph [CK1966], but the results in [CK1966] were stated with the unnecessary hypothesis
that the general structures have  a predicate symbol for equality.
Using Fact \ref{f-metric-theory}, it will follow as a corollary that these results also hold for metric structures.
These consequences for metric structures were also proved in the paper [F], by adapting the classical proofs of the corresponding first order results.

The following lemma is the analogue for continuous model theory of Lemma 3.2.1 in [CK2012].

\begin{lemma}  \label{l-preserved}  Let  $\Gamma$ be a set of sentences that is closed under the application of the $\min$ connective and unary increasing connectives.  The following are equivalent.

\begin{itemize}
\item[(i)]  $T$ is $S$-equivalent to a set of sentences $U\subseteq\Gamma$.
\item[(ii)]  If $\cu{M}, \cu{N}$ are general models of $S$, $\cu{M}$ is a general model of $T$, and every sentence $\gamma\in\Gamma$ that is true in $\cu{M}$ is true in $\cu{N}$, then $\cu{N}$ is a general model of $T$.
\end{itemize}
\end{lemma}

\begin{proof}  It is trivial that (i) implies (ii).  Assume (ii), and let $U$ be the set of all sentences $\gamma\in\Gamma$ such that every general model of $S\cup T$ is a general model of $\gamma$.  Then every general model of $S\cup T$ is a general model of $S\cup U$.  Let $\cu{N}$ be a general model of $S\cup U$.  Consider sentences $\psi_0,\ldots,\psi_n\in \Gamma$ and numbers $r_0,\ldots,r_n\in[0,1]$.  Using the fact that $\Gamma$ is closed under $\min$ and increasing unary connectives, one can show that $\Gamma$ contains a sentence $\gamma$ saying that $\psi_i\le r_i$ for some $i\le n$.  Now let $\varepsilon>0$ and put $r_i=\psi_i^\cu{N}\dotminus\varepsilon$ for each $i\le n$.    Then $\gamma$ is not true in $\cu{N}$.  Since $\cu{N}$ is a general model of $S\cup U$, $\gamma\notin U$, so $\gamma$ cannot be true in every general model of $S\cup T$.  Hence there is a general model $\cu{M}$ of $S\cup T$ such that $\psi_i^\cu{M}\ge \psi_i^\cu{N}\dotminus\varepsilon$ for all $i\le n$.  By the compactness theorem, there is a general model $\cu{M}$ of $S\cup T$ such that  $\psi^\cu{M}\ge\psi^\cu{N}$ for every  sentence $\psi\in\Gamma$.  Then by (ii), $\cu{N}$ is a general model of $T$, so (i) holds.
\end{proof}

Here is an analogue of Lemma \ref{l-preserved} for metric theories.

\begin{cor}  Suppose $L$ is a metric signature over $V$, and $S, T$ are metric theories with signature $L$.
Let  $\Gamma$ be a set of sentences that is closed under the application of the $\min$ connective and unary increasing connectives.  The following are equivalent.

\begin{itemize}
\item[(i)]  $T$ is $S$-equivalent to a set of sentences $U\subseteq\Gamma$.
\item[(ii)]  If $\cu{M}_+, \cu{N}_+$ are pre-metric models of $S$, $\cu{M}_+$ is a pre-metric model of $T$, and every sentence $\gamma\in\Gamma$ that is true in $\cu{M}_+$ is true in $\cu{N}_+$, then $\cu{N}_+$ is a pre-metric model of $T$.
\end{itemize}

\end{cor}

\begin{proof}  This is just a restatement of Lemma \ref{l-preserved} in the special case that $S, T\models\met(L)$.
\end{proof}
Throughout this appendix, $S$ and $T$ denote continuous theories.  The set of \emph{existential formulas} is defined as the least set of formulas that contains all quantifier-free formulas and is closed under the application of increasing connectives and the $\inf$ quantifier.

\begin{lemma} \label{l-special-embedding}  (Theorem 7.2.8 in [CK1966].) Suppose that $\cu{N}$ is special, $|M|\le|N|$, and every existential sentence that is true in $\cu{M}$ is true in $\cu{N}$.  Then $\cu{M}$ is embeddable in $\cu{N}$.
\end{lemma}

\begin{proof}  Arrange the elements of $M$ in a sequence $\langle a_\alpha\colon \alpha<|N|\rangle$.  By transfinite induction, build a sequence
$\langle b_\alpha\colon \alpha<|N|\rangle$ of elements of $N$ such that every existential formula that is true for a tuple of $a_\alpha$'s is true for the corresponding tuple of $b_\alpha$'s.  Then show that the mapping $a_\alpha\mapsto b_\alpha$ is an embedding of $\cu{M}$ into $\cu{N}$.
\end{proof}

\begin{thm} \label{t-existential}  (Theorem 7.2.11 in [CK1966].)   The following are equivalent.

\begin{itemize}
\item[(i)] $T$ is $S$-equivalent to a set of existential sentences.
\item[(ii)] For all general models $\cu{M}, \cu{N}$ of $S$, if $\cu{M}$ is embeddable in $\cu{N}$, and $\cu{M}$ is a general model of $T$, then $\cu{N}$ is a general model of $T$.
\end{itemize}
\end{thm}

\begin{proof}  Assume (i). Suppose $h$ is an embedding from $\cu{M}$ into $\cu{N}$, and $\cu{M}$ is a model of $T$.  One can show by induction on complexity that for every existential formula $\varphi(\vec{x})$ and tuple $\vec{a}$ in $M$, $\varphi^\cu{M}(\vec{a})\ge \varphi^\cu{N}(h(\vec{a}))$.  It follows that $\cu{N}$ is a model of $T$, so (ii) holds.

Assume (ii).  We apply Lemma \ref{l-preserved} where $\Gamma$ is the set of existential sentences.  Note that $\Gamma$ is closed under the application of $\min$ and of increasing unary connectives.  Suppose $\cu{M}, \cu{N}$ satisfy the hypotheses of Lemma \ref{l-preserved} (ii), that is,
$\cu{M}, \cu{N}$ are general models of $S$, $\cu{M}$ is a general model of $T$, and every sentence $\gamma\in\Gamma$ that is true in $\cu{M}$ is true in $\cu{N}$.   By Fact \ref{f-special-exist}, we may assume that $\cu{N}$ is a special structure and $|M|\le|N|$.  Then by Lemma \ref{l-special-embedding}, $\cu{M}$ is embeddable in $\cu{N}$.  Therefore by Condition (ii) above, $\cu{N}$ is a general model of $T$.  Hence by Lemma \ref{l-preserved}, Condition (i) above holds.
\end{proof}

The set of \emph{universal formulas} is defined as the least set of formulas that contains all quantifier-free formulas and is closed under the application of increasing connectives and the $\sup$ quantifier.

\begin{exercise} \label{e-universal}    The following are equivalent.

\begin{itemize}
\item[(i)] $T$ is $S$-equivalent to a set of universal sentences.
\item[(ii)] For all general models $\cu{M}, \cu{N}$ of $S$, if $\cu{M}$ is embeddable in $\cu{N}$, and $\cu{N}$ is a general model of $T$, then $\cu{M}$ is a general model of $T$.
\end{itemize}
\end{exercise}

The set of \emph{universal-existential formulas} is defined as the least set of formulas that contains all existential formulas and is closed under the application of increasing connectives and the $\sup$ quantifier.

\begin{thm} \label{t-chains} (Exercise 7F in [CK1966].)  The following are equivalent:

\begin{itemize}
\item[(i)]  $T$ is $S$-equivalent to a set of universal-existential sentences.
\item[(ii)]  For every increasing chain $\cu{M}_0\subseteq \cu{M}_1\subseteq\cdots$ of general models of $S\cup T$, if $\bigcup_n \cu{M}_n$ is a general model of $S$ then it is a general model of $T$.
\item[(iii)] If  $\cu{M}, \cu{M}', \cu{N}$ are general models of $S$ such that $\cu{M}\subseteq \cu{N}\subseteq \cu{M}'$, $\cu{M}\prec \cu{M}'$, and $\cu{N}$ is a general model of $T$, then $\cu{M}$ is a general model of $T$.

\end{itemize}
\end{thm}

\begin{proof}   Assume (i) and let $\cu{M}_0\subseteq \cu{M}_1\subseteq\cdots$ satisfy the hypotheses of (ii).  By induction on the complexity of formulas, for every assignment of the free variables, every universal-existential formula that is true in all but finitely many $\cu{M}_n$ is true in $\bigcup_n \cu{M}_n$.  Hence (ii) holds.

Assume (ii), and let $\cu{M}, \cu{M}', \cu{N}$ satisfy the hypotheses of (iii).  Let $\kappa$ be a strong limit cardinal greater than $|V|$.  By Fact \ref{f-special-exist}, there are
special models $\cu{M}_0\equiv \cu{M}, \cu{M}_1\equiv \cu{M}, \cu{N}_0\equiv \cu{N}$ of cardinality $\kappa$ that satisfy the hypotheses of (iii).  By Fact \ref{f-special-unique} there is an increasing chain $\cu{M}_0\subseteq \cu{N}_0\subseteq \cu{M}_1\subseteq \cu{N}_1\subseteq\cdots$ of special structures of cardinality  $\kappa$ such that for each $n$, $\cu{M}_n\equiv \cu{M}, \cu{M}_n\prec \cu{M}_{n+1}$, and $\cu{N}_n\equiv \cu{N}$.
Then $\cu{M}_0\prec\bigcup_n \cu{M}_n=\bigcup_n \cu{N}_n$.  Then for each $n$, $\cu{N}_n$ is a model of $S\cup T$ and $\bigcup_n \cu{N}_n$ is a model of $T$, so by (ii), $\bigcup_n \cu{N}_n$ is a model of $T$.  Finally, since $\cu{M}\equiv \bigcup_n \cu{N}_n$, $\cu{M}$ is a general model of $T$ and (iii) holds.

Assume (iii).  To prove (i), we apply Lemma \ref{l-preserved} where $\Gamma$ is the set of universal-existential sentences.  As before, $\Gamma$ is closed under the application of $\min$ and of increasing unary connectives.  This time, we suppose $\cu{N}, \cu{M}$ satisfy the hypotheses of Lemma \ref{l-preserved}, and show that $\cu{M}$ is a general model of $T$.  We may assume that $\cu{N}$ is a special structure, $\cu{M}$ is reduced, and $|M|<|N|$.  As in the proof of Lemma \ref{l-special-embedding}, one can show that there is an embedding $h\colon \cu{M}\to \cu{N}$ such that for any universal-existential formula $\varphi(\vec{x})$ and tuple $\vec{a}$ in $M$, $\varphi^\cu{N}(h(\vec{a}))\ge \varphi^\cu{M}(\vec{a})$.  Then every existential sentence that is true in the expanded structure $(\cu{N},h(a)\colon a\in M)$ is true in $(\cu{M},a\colon a\in M)$.  By Fact \ref{f-special-exist}, $\cu{M}$ has a special elementary extension $\cu{M}'$ of cardinality $|M'|\ge |N|$, and by Lemma \ref{l-special-embedding}, there is an embedding
$$k\colon (\cu{N},h(a)\colon a\in M)\to(\cu{M}',a\colon a\in M).$$
Then there are isomorphic copies $\cu{M}_0,\cu{M}'_0$ of $\cu{M},\cu{M}'$ such that $\cu{M}_0\subseteq \cu{N}\subseteq \cu{M}'_0$ and $\cu{M}_0\prec \cu{M}'_0$.  By (iii), $\cu{M}_0$ is a general model of $T$, so $\cu{M}$ is a general model of $T$, as required.
\end{proof}

By Fact \ref{f-metric-theory}, $\met(L)$ is a set of universal sentences.  It follows that if $\cu{M}$ is embeddable in a pre-metric structure then $\cu{M}$ is a pre-metric structure, and also that the union of any chain of pre-metric structures is a pre-metric structure.

The set of \emph{positive formulas} is the least set of formulas that contains the set of atomic formulas and is closed under the application of increasing connectives and the quantifiers $\sup$ and $\inf$.  Given two general structures $\cu{M},\cu{N}$, a \emph{homomorphism} from $\cu{M}$ into $\cu{N}$ is a function $h$ from $M$ onto $N$ such that $h(c^\cu{M})=c^\cu{N}$ for each constant symbol $c\in V$, and for every $n$ and $\vec{a}\in M^n$, $h(F^\cu{M}(\vec{a}))=F^\cu{N}(h(\vec{a}))$ for every function symbol $F\in V$ of arity $n$, and $P^\cu{M}(\vec{a})\ge P^\cu{N}(h(\vec{a}))$ for every predicate symbol $P\in V$ of arity $n$.  We say that $\cu{N}$ is a \emph{homomorphic image} of $\cu{M}$ if there is a homomorphism from $\cu{M}$ onto $\cu{N}$.  Note that every homomorphic image of a pre-metric structure is a pre-metric structure.

\begin{lemma}  \label{l-special-homomorphic}  (Theorem 7.3.7 in [CK1966].)  Suppose that $\cu{M}, \cu{N}$ are special, that either $N$ is finite or $|M|=|N|$, and that every positive sentence that is true in $\cu{M}$ is true in $\cu{N}$. Then $\cu{N}$ is a homomorphic image of $\cu{M}$.
\end{lemma}

\begin{proof}  Similar to the proof of Lemma \ref{l-special-embedding} but using a back-and-forth construction.
\end{proof}

\begin{thm}  \label{t-homomorphic}
(Theorem 7.3.9 in [CK1966].)    The following are equivalent.

\begin{itemize}
\item[(i)] $T$ is $S$-equivalent to a set of positive sentences.
\item[(ii)] For all general models $\cu{M}, \cu{N}$ of $S$, if $\cu{N}$ is a homomorphic image of $\cu{M}$, and $\cu{M}$ is a general model of $T$, then $\cu{N}$ is a general model of $T$.
\end{itemize}
\end{thm}

\begin{proof}  Similar to the proof of Theorem \ref{t-existential}.
\end{proof}

The set of \emph{positive existential formulas} is the least set of formulas that contains the set of atomic formulas and is closed under the application of increasing connectives and the $\inf$ quantifier.

\begin{exercise} \label{e-positive-existential}  (Theorem 7.2.11 in [CK1966].)   The following are equivalent.

\begin{itemize}
\item[(i)] $T$ is $S$-equivalent to a set of positive existential sentences.
\item[(ii)] For all general models $\cu{M}, \cu{N}$ of $S$, if there is a homomorphic embedding of $\cu{M}$ in $\cu{N}$, and $\cu{M}$ is a general model of $T$, then $\cu{N}$ is a general model of $T$.
\end{itemize}
\end{exercise}

\begin{cor}  Let $L$ be a metric signature over $V$. If $S$ contains $\met(L)$, then Theorems \ref{t-existential}, \ref{t-chains}, and \ref{t-homomorphic}, and Exercises \ref{e-universal} and \ref{e-positive-existential},  still hold when all structures mentioned are taken to be metric structures with signature $L$.
\end{cor}

\begin{proof} In view of Remark \ref{r-saturated-metric}, the proofs go through using metric structures instead of general structures.
\end{proof}

\section{The Keisler-Shelah Theorem for General Structures}

In this section, we use the ideas developed in Section \ref{s-elim-CH} to prove the Keisler-Shelah Theorem ([Sh]) for general structures, an idea suggested to us by James Hanson:

\begin{thm}\label{genKS}
Suppose that $V$ is a vocabulary and $\cu{M}$ and $\cu{N}$ are elementarily equivalent $V$-structures.  Then there is an ultrafilter $\cu{D}$ over a set $I$ such that $\cu{M}^I/{\cu{D}}$ and $\cu{N}^I/{\cu{D}}$ are isomorphic.
\end{thm}

We remark that no proof of the Keisler-Shelah theorem for continuous logic (in its current incarnation) has appeared in the literature thus far.  (A proof in the context of \emph{positive bounded logic}, a predecessor of continuous logic, can be found in [HI].)

Before proving Theorem \ref{genKS}, we need a lemma:

\begin{lemma}
Suppose that $\cu{M}$ and $\cu{N}$ are $\aleph_1$-saturated elementarily equivalent $V$-structures.  Then $\cu{M}_\downarrow$ and $\cu{N}_\downarrow$ are elementarily equivalent.
\end{lemma}

\begin{proof}
As shown in Lemma 2.4 of [GH], the assumption of the current lemma implies that, for any $k\in \mathbb N$, player II has a strategy for winning the strengthening of the usual Ehrenfeucht-Fra\"isse game between $\cu{M}$ and $\cu{N}$ of length $k$, where winning means that, denoting by $a_1,\ldots,a_k$ and $b_1,\ldots,b_k$ the elements of $M$ and $N$ played during the game, the map $a_i\mapsto b_i$ induces an isomorphism between the substructures that the tuples generate respectively.  By playing according to this winning strategy, player II can win any ordinary Ehrenfeuch-Fra\"isse game between $\cu{M}_\downarrow$ and $\cu{N}_\downarrow$, whence $\cu{M}_\downarrow$ and $\cu{N}_\downarrow$ are elementarily equivalent.
\end{proof}

We invite the reader to verify that the preceding lemma fails when the saturation assumption is removed.  We are now ready to prove the main theorem of this section:

\begin{proof}[Proof of Theorem \ref{genKS}]
Suppose that $\cu{M}$ and $\cu{N}$ are elementarily equivalent $V$-structures.  By replacing $\cu{M}$ and $\cu{N}$ with ultrapowers respect to a nonprincipal ultrafilter on $\mathbb N$, we may assume that they are $\aleph_1$-saturated.  By the previous lemma, we have that $\cu{M}_\downarrow$ and $\cu{N}_\downarrow$ are elementarily equivalent.  By the classical Keisler-Shelah Theorem in [Sh], we have an ultrafilter $\cu{D}$ over a set $I$ such that $(\cu{M}_\downarrow)^I/{\cu{D}}$ and $(\cu{N}_\downarrow)^I/{\cu{D}}$ are isomorphic, whence so are $((\cu{M}_\downarrow)^I/{\cu{D}})_\uparrow$ and $((\cu{N}_\downarrow)^I/{\cu{D}})_\uparrow$.  By Lemmas \ref{l-downup} and \ref{p-down-reduced-product}, these aforementioned structures are simply $\cu{M}^I/{\cu{D}}$ and $\cu{N}^I/{\cu{D}}$, as desired.
\end{proof}

\section*{References}

[Be1]  Itai Ben Yaacov.  Positive Model Theory and Compact Abstract Theories.  Journal of Mathematical Logic 3 (2003), 85-118.

[Be2]  Ita\"i{} Ben Yaacov.  On Theories of Random Variables.  Israel J. Math 194 (2013), 957-1012.

[BBHU]  Ita\"i{} Ben Yaacov, Alexander Berenstein,
C. Ward Henson and Alexander Usvyatsov. Model Theory for Metric Structures.  In Model Theory with Applications to Algebra and Analysis, vol. 2,
London Math. Society Lecture Note Series, vol. 350 (2008), 315-427.

[Ca] Xavier Caicedo.  Maximality of Continuous Logic.  Pp. 105-130 in Beyond First Order Model Theory, edited by Jose Iovino, CRC Press (2017),

[CK1966]  C.C.Chang and H. Jerome Keisler.  Continuous Model Theory.  Annals of Mathematics Studies, Princeton 1966.

[CK2012]  C.C.Chang and H. Jerome Keisler.  Model Theory.  Dover 2012.

[Ga]  Fred Galvin.  Horn Sentences.  Annals of Mathematical Logic 1 (1970), pp. 389-422.

[Gh]  Saeed Ghasemi.  Reduced products of Metric Structures: A Metric Feferman-Vaught Theorem. Journal of Symbolic Logic 81 (2016), pp. 856-875.

[GH] Isaac Goldbring and Bradd Hart.  On the theories of McDuffs II$_1$ factors.  International Mathematics Research Notices 27 (2017), pp. 5609-5628.

[He]  C. Ward Henson.  Nonstandard Hulls in Banach Spaces, Israel J. Math. 25 (1976), 108-144.

[HI]  C. Ward Henson and Jose Iovino.  Ultraproducts in Analysis, in Analysis and Logic, London Math. Soc. Lecture Note Series 262 (2002), 1-113.

[FS]  Ilijas Farah and Saharon Shelah.  Rigidity of Continuous Quotients.  J. Math. Inst. Jussieu, 15 (2016), 1-28.

[Ke]  H. Jerome Keisler.  Model Theory for Real-valued Structures.  Available online at www.math.wisc.edu/$~$keisler.

[Ke65]  H. Jerome Keisler.  Reduced products and Horn classes Trans. Am. Math. Soc. 117 (1965), 307-328.

[Lo] Vinicius C. Lopes.  Reduced Products and Sheaves of Metric Structures, Math. Log. Quart. 59 (2013), 219-229.

[Sh]  Saharon Shelah.  Every Two Elementarily Equivalent Models Have Isom9orphic Ultrapowers.  Israel Journal of Mathematics 10 (1972), 224-233.

\end{document}